\documentclass[11pt]{article}
\usepackage{enumitem}
\usepackage[font=small,format=plain,labelfont=bf,up]{caption}
\usepackage[a4paper,nohead,ignorefoot,%
twoside,scale=0.852,
twoside,scale=0.86, %
bindingoffset=1.25cm]{geometry}
\usepackage{amssymb,amsfonts,amsmath,amsthm}
\usepackage[]{graphicx}
\usepackage{color}
\newcommand{\url}[1]{#1} 
\usepackage{hyperref}%
\definecolor{gray}{rgb}{0.2,0.2,.2}
\hypersetup{%
  unicode=true,          
  colorlinks=true,       
  linkcolor=gray,          
  citecolor=gray      
}
\DeclareMathOperator{\sgn}{\mathrm{sgn}}

\newcommand{\fspace}[1]{{\mathsf{#1}}}
\newcommand{\fspaceL}{\fspace{L}}

\newcommand{\ol}[1]{{\overline{#1}}}

\newcommand{\Rset}{{\mathbb{R}}}

\newcommand{\Nset}{{\mathbb{N}}}

\newcommand{\cointerval}[2]{[#1,\,#2)}%
\newcommand{\oointerval}[2]{(#1,\,#2)}%
\newcommand{\ccinterval}[2]{[#1,\,#2]}%

\newcommand{\DO}[1]{{O\at{#1}}}

\newcommand{\odd}{{\rm \,odd}}
\newcommand{\even}{{\rm \,even}}

\newlength{\mhpicDwidth}
\newlength{\mhpicDvsep}
\newlength{\mhpicDhsep}

\newlength{\mhpicPwidth}
\newlength{\mhpicPvsep}
\newlength{\mhpicPhsep}

\newlength{\mhpicWhsep}

\newcommand{\pair}[2]{{\left({#1},\,{#2}\right)}}

\newcommand{\bskp}[2]{{\big\langle{#1},\,{#2}\big\rangle}}

\newcommand{\at}[1]{{\left({#1}\right)}}

\newcommand{\bat}[1]{{\big(#1\big)}}
\newcommand{\Bat}[1]{{\Big(#1\Big)}}

\newcommand{\ul}[1]{\underline{#1}}

\newcommand{\bigpar}{\par\quad\newline\noindent}

\newcommand{\norm}[1]{\left\|{#1}\right\|}
\newcommand{\bnorm}[1]{\big\|{#1}\big\|}

\newcommand{\nnorm}[1]{\|{#1}\|}
\newcommand{\abs}[1]{\left|{#1}\right|}
\newcommand{\babs}[1]{\big|{#1}\big|}

\newcommand{\dint}[1]{\,\mathrm{d}#1}

\newcommand{\al}{{\alpha}}

\newcommand{\ka}{{\kappa}}
\newcommand{\la}{{\lambda}}
\newcommand{\si}{{\sigma}}

\newcommand{\calA}{\mathcal{A}}

\newcommand{\calF}{\mathcal{F}}

\newcommand{\calL}{\mathcal{L}}

\newcommand{\calN}{\mathcal{N}}

\newcommand{\calU}{\mathcal{U}}

\theoremstyle{plain}
\newtheorem{theorem}{Theorem}[]

\newtheorem{lemma} [theorem]{Lemma}
\newtheorem{proposition}[theorem]{Proposition}
\newtheorem{conjecture}[theorem]{Conjecture}

\newtheorem{mresult}[theorem]{Main Result}

\newtheorem{assumption} [theorem]{Assumption}

\theoremstyle{definition}

\usepackage{float}
\begin{document}
%
%
\title{A uniqueness result for a simple superlinear eigenvalue problem}
\date{\today}
\author{%
Michael Herrmann%
\footnote{Technische Universit\"at Braunschweig, Germany, {\tt michael.herrmann@tu-braunschweig.de}}
\and
Karsten Matthies%
\footnote{University of Bath, United Kingdom, {\tt k.matthies@bath.ac.uk}}
} %
\maketitle
%
%
%
%
\vspace{-0.025\textheight}
\begin{abstract}
We study the eigenvalue problem for a superlinear convolution operator in the special case of bilinear constitutive laws and establish the existence and uniqueness of a one-parameter family of nonlinear eigenfunctions under a topological shape constraint. Our proof uses a nonlinear change of scalar parameters and applies Krein-Rutmann arguments to a linear substitute problem. We also present numerical simulations and discuss the asymptotics of two limiting cases.
\end{abstract}
%
%
%
\quad\newline\noindent%
\begin{minipage}[t]{0.15\textwidth}%
 Keywords:
\end{minipage}%
\begin{minipage}[t]{0.8\textwidth}%
\emph{nonlinear eigenvalue problems}, \emph{nonlocal coherent structures},\\ \emph{Krein-Rutmann theorems},   \emph{asymptotic analysis of nonlinear integral operators}
\end{minipage}%
\medskip
\newline\noindent
\begin{minipage}[t]{0.15\textwidth}%
MSC (2010): %
\end{minipage}%
\begin{minipage}[t]{0.8\textwidth}%
45G10,  
45M05, 
47J10 	
\end{minipage}%
%
%
%
\setcounter{tocdepth}{3}
\setcounter{secnumdepth}{3}{\scriptsize{\tableofcontents}}
%
%
\section{Introduction}\label{sect:intro}
%
%
Nonlinear analogues to linear eigenvalue problems arise in many branches of the sciences and often model coherent structures in spatially extended dynamical systems. Examples will be discussed below in greater detail and include traveling waves in Hamiltonian lattices as well as certain nonlocal aggregation models. Nonlinear eigenvalue equations often combine nonlinear superpositions with linear pseudo-differential operators, where the latter can represent convolutions, spatial derivatives, or the solution of an elliptic auxiliary problem.
\par
A simple and spatially one-dimensional prototype is the scalar equation 
\begin{align}
\label{Eqn:EV1}
\sigma \, u = a\ast f\at{u}
\end{align}
which involves a localized convolution kernel $a$ and a given nonlinear function $f$. The problem is to find solution pairs $\pair{\si}{u}$ consisting of an real eigenvalue $\si$ and a nontrivial scalar eigenfunction $u$ which might be of  periodic or homoclinic type. As explained below, the existence of solutions to \eqref{Eqn:EV1} can be established by many different approaches and this has already been done for special choices or certain classes of $a$ and $f$. The uniqueness problem, however, is notoriously difficult in the superlinear case and we are not aware of any rigorous result that applies to \eqref{Eqn:EV1}.
\par
As a first step towards a more general uniqueness theory we study in this paper the case of bilinear functions $f$ which are piecewise linear with two slopes. This simplifying assumption allows us to transform \eqref{Eqn:EV1} into the linear eigenvalue problem for a modified convolution operator with cut-off parameter $\xi$.  We further suppose that both the kernel $a$ and the eigenfunction $u$ are nonnegative, even, and unimodal. This shape constraint provides an existence and uniqueness result for the linear substitute problem thanks to a variant of the classical Krein-Rutmann theorem. Moreover, it ensures that the nonlinear relation between $\si$ and $\xi$ is bijective. The combination of both arguments implies an uniqueness result as well as a novel existence proof for the nonlinear eigenvalue problem \eqref{Eqn:EV1}.
\par
Before we specify our assumptions and findings, we continue with an informal overview on possible applications and give a more detailed discussion of the underlying mathematical problems. 
%
%
%
\paragraph{Hamiltonian lattice waves}
%
Traveling waves in Hamiltonian lattices can be viewed as the nonlinear analogues to plane-wave excitations and provide the building blocks for more complex solution patterns. The prototypical example of such spatially discrete waves propagate in Fermi-Pasta-Ulam-Tsingou (FPUT) chains with nearest neaghbor interactions and comply with the advance-delay differential equation
\begin{align}
\label{Eqn:FPU}
\si\, u^{\prime\prime}\at{x}=f\bat{u\at{x+1}}+f\bat{u\at{x-1}}-2\,f\bat{u\at{x}}\,,
\end{align}
see for instance \cite{FW94, FV99, Her10} for more background information. Here, $\si$ is the squared wave speed and $u$ denotes the unknown profile function for the atomic distances. Moreover, $x$ stands for the spatial variable in the comoving frame and $f$ abbreviates the nonlinear stress-strain relation, which is usually be given by the derivative of an  interaction potential. After twofold integration with respect to $x$ --- and eliminating the constants of integration by assuming a homoclinic wave profile --- equation \eqref{Eqn:FPU} can be transformed into the eigenvalue problem \eqref{Eqn:EV1} with a tent-map kernel, see equation \eqref{Eqn:Kernel2} below. Another, but closely related, example is the first order equation
\begin{align}
\label{Eqn:DCL}
\si \,u^\prime\at{x} = f\bat{u\at{x+1}}-f\bat{u\at{x-1}}\,,
\end{align}
which describes the traveling waves in a Hamiltonian semidiscretization of scalar conservation laws, see \cite{Her12a}. It is likweise equivalent to \eqref{Eqn:EV1}, but now equipped with the piecewise constant kernel \eqref{Eqn:Kernel3}.
\par 
For FPUT chains, the existence of periodic or solitary waves has been established in different frameworks and the employed methods include critical point techniques \cite{Pan05}, constrained optimization \cite{FW94,FV99,Her11b,HM20}, and asymptotic or perturbative arguments \cite{FP99,IJ05,Jam12}. More recent work in this field concerns waves in lattice systems with a finite or even infinite number of interaction bonds and vector-valued displacement fields. The corresponding advance-delay differential equations can also be written in form of nonlinear eigenvalue problems but include more than one convolution kernel and several nonlinearities, see \cite{HML16,ChH18,PV18,HM19a}. 
\par
Although the existence theory of Hamiltonian lattice waves developed quite well, very little is known about their uniqueness, parameter dependence, and dynamical stability. The only rigorous results either concern completely integrable cases (Toda chain) or are restricted to asymptotic regimes, in which the problem can be tackled as a singular perturbation of an underlying ODE regime, see \cite{FP99} and \cite{HM19b} for the limits of small and large wave speeds, respectively. Any progress in this direction would enhance our understanding of energy transport in quasilinear Hamiltonian systems with strong dispersion. Of particular interest are the lattice variants of Whitham's modulation theory and their predictions on the creation and propagation of dispersive shock waves in spatially discrete media. For instance, for FPUT chains with superquadratic force function $f$ we expect the existence of a 4-parameter family of periodic and linearly stable lattice waves, which can be modulated according to a hyperbolic system of 4 nonlinear conservation laws, see for instance \cite{FV99,DHM06,DH08,YChYK17}.  There exists strong numerical evidence for the existence of modulated traveling waves but we still lack a rigoros understanding for non-integrable cases. An important subproblem in this context is to characterize the solution set of \eqref{Eqn:EV1} for at least special kernels and the spectral properties of its linearization. 
\par
We further emphasize that  standing waves or breather solutions in lattice system of coupled oscillators can also be linked to nonlinear eigenvalue problems although the corresonding equations are often semilinar and involve discrete convolution operators. A typical example is
\begin{align*}
\si\, u_j = \al\,\bat{u_{j+1}+u_{j-1}}+f\at{u_j}
\end{align*}
and appears in the theory of discrete nonlinear Schr\"odinger equations. Solutions can be constructed by variational techniques or asymptotic methods, see for instance \cite{MKA94,Wei99,Her11b} and the overview in \cite{Kev09}, but it seems that there is no global theory concerning the uniqueness of solutions for the relevant classes of nonlinearieties (which migth be either focussing or defocussing). 
%
%
\paragraph{Further related applications}
%
%
Many aggregation models in mathematical biology also involve both convolution operators and nonlinearities. A first example is the nonlocal parabolic PDE
\begin{align*}
\partial_t \varrho = \Delta_{x} \bat{ h\at{a\ast \varrho } \varrho }
\end{align*}
which has been introduced and studied in \cite{BHW13,HO15} to model the pattern forming effect of concentration dependent diffussion mobilities. The scalar function $h$, which is supposed to be nonnegative and decreasing, describes that the tendency of biological individuals to undertake a random walk depends on the local population density $\varrho\pair{t}{x}$, where the convolution kernel reflects the corresponding interaction domain. Despite its apparent simplicity, the equation exhibits a rather intruiging dynamics even in one space dimension. Starting with random initial data one oberserves the rather rapid formation of metastable peaks, which in turn interact and annihilate each other in a coarsening dynamics on much larger time scales. Each localized peak can be viewed as a stationary state and the corresponding equation $h\at{a\ast \varrho } \varrho =c$ is equivalent to the eigenvalue problem \eqref{Eqn:EV1} via the substitution $u=a\ast \varrho$, $f\at{u}=1/h\at{u}$, and $\si=1/c$. It is desirable to  understand the complete set of metastable states and their spectral properties with respect to the linearized parabolic dynamics.
\par
A second example in one space dimension is the Wasserstein gradient flow
\begin{align}
\label{Eqn:BAM2}
\partial_t \varrho = \partial_{x} \Bat{ \varrho\,\partial_x
\bat{g\at{\varrho}-a\ast \varrho 
}
 }\,,
\end{align}
where $g$ is now a strictly increasing function, see  \cite{BFH14,Kai17} and references therein. The nonlinear eigenvalue problem \eqref{Eqn:EV1} yields again stationary states with $u=g\at{\varrho}+c=a\ast \varrho$ provided that $\si^{-1}f$ is the inverse function to $g+c$. However, there further exist more general steady states $\varrho$ which are compactly supported and satify the nonlinear integral equation  on that support only, see \cite{BFH14}.  
\par
The scalar and one-dimensional equation \eqref{Eqn:EV1} can also be viewed as a simplified toy model for other nonlinear eigenvalue problems. For instance, chimera states in the Kuramoto equation can be characterized by a complex-valued analogue to \eqref{Eqn:EV1} as explained in \cite{OMT08,Ome18}. Further examples are the nonlocal variants of the Allan-Cahn equation or systems of reaction diffusion equations. In this context, however, the nonlinearities are typically bistable instead of monotone and the equation for the relevant coherent states often involves additional derivative terms that account for a possible progagation with constant speed, see for instance \cite{FS15,BS18}. The existence and uniqueness problem is hence more challenging and has already been investigated by perturbative arguments as described in \cite{AFSS16,ST19}. Finally, many semilinear elliptic PDE can be interpreted as nonlinear eigenvalue problems, where $a$ is the order preserving solution operator of the underlying linear problem. The standard example in one space dimension is the exponential kernel  $a\at{x}=\tfrac12\exp\at{-\abs{x}}$, which can be viewed as the fundamental solution to the differential operator $-\partial_x^2 +1$. The corresponding eigenvalue problem \eqref{Eqn:EV1} can be tackled by planar ODE techniques applied to $-\si u^{\prime\prime}=f\at{u}-\si\,u$ and this provides a complete characterization of all solutions.  For general kernels $a$, however, we lack such a simplyfying reinterpretation. Notice also that the space variable $x$ in  elliptic PDEs is often confined to a certain subdomain of $\Rset^n$,  which requires to impose additional boundary conditions for $u$. 
%
%
\paragraph{Problems, discussion, and outlook}
%
A complete mathematical theory of nonlinear eigenvalue problems should address the existence, uniqueness, and parameter dependence of  solutions but also the spectral properties of its linearization as those naturally appear in the corresponding continuation and bifurcation analysis.  All available rigorous and heuristic results indicate that the mathematical theory strongly depends on the structural properties of the convolution kernel and the nonlinearities. 
\par
In this paper we restrict our considerations to one space dimension, unimodal kernels $a$, and a superlinear function $f$. This setting is relevant for both Hamiltonian lattice waves and biological aggregation models and still related to many unsolved mathematical problems concerning the uniqueness and dynamical stability of solutions.   The case of sublinear functions $f$ should be simpler because both existence and uniquness results can be inferred from the abtract theory in \cite{Rab71}. The case of non-monotone functions $f$ is of course much more involved and will not be adressed  here. Furthermore, we do not discuss general convolution kernels (being nonnegative or not) or the case of vector-valued eigenfunctions.
\par
The existence of periodic or solitary solutions to \eqref{Eqn:EV1} can be established by various methods and we already mentioned variational and asymptotic techniques for special classes of $a$ and $f$. For sufficiently nice kernels $a$, the existence of periodic solutions can also be deduced from the classical Crandall-Rabinowitz theory \cite{CR71}, which starts with a periodic solution to the linear equation 
\begin{align*}
\si_0\, u_0= f^\prime\at{0}\, a\ast u_0
\end{align*}
and provides a global bifurcation result under natural nondegenericity conditions. In one space dimension, all existence conditions can easily be checked by means of Fourier arguments. We are, however, not aware of any related uniqueness result that covers superlinear functions $f$ and is not based on local bifurcation analysis. 
\par
Motivated by numerical simulations and the existing results for special cases we conjecture --- both for any fixed periodicty length $L<\infty$ and the solitary limit $L=\infty$ --- the existence of a unique family of solutions $\pair{\si}{u}$ that can be parametrized by the eigenvalue $\si$ and involves eigenfunctions $u$ that are nonnegative, even, and unimodal (see below for a precise definition). This statement does not exclude the existence of further eigenfunctions that do not meet the shape constraint and can be interpreted as a nonlinear version of the famous Krein-Rutmann theorem (or its Perron-Frobenius analogue for matrices). There already exist some nonlinear variants of these theorems, see for instance \cite{Mah07,Ara18}, but these are restricted to $1$-homogeneous nonlinearities which are necessarily linear in the scalar case considered here.
\par
Moreover, except for the exponential kernel it is not clear how to characterize all solutions $\pair{\si}{u}$ to the linearized eigenvalue problem 
\begin{align}
\label{Eqn:EVLin}
\si_*\, u + \si\, u_* = a\ast\bat{f^\prime\at{u_*} u}\,,
\end{align}
where $\pair{\si_*}{u_*}$ denotes a fixed solution to the nonlinear equation. Preliminary numerical simulations indicate that there should be some analogue to the Sturm-Liouville theory for second order ODEs but it seems that there is no rigorous result in this direction that applies to arbitrary kernels $a$.
\par
With the present paper we wish to contribute to the general theory of the eigenvalue problem \eqref{Eqn:EV1} with superlinear $f$. At the moment we are only able to prove our conjecture for piecewise linear functions $f$ but hope to investigate more general nonlinearities in a forthcoming study. The next natural step is to allow $f^\prime$ to exhibit more than one jump discontinuity but even the trilinear case is, as far as we see, considerably more involved than the bilinear situation studied here. We also postpone the investigation of the linearized eigenvalue problem \eqref{Eqn:EVLin} or related equations to future work.
%
%
\paragraph{Assumptions and results}
%
%
As illustrated Figure \ref{Fig:Nonl}, we consider functions $f$ that are piecewise linear on the interval $\cointerval{0}{\infty}$. More precisely, we set
\begin{align}
\label{Eqn:Nonl}
f\at{r}:=\left\{
\begin{array}{lccl}
\zeta \, r&& \text{for}& r\in\ccinterval{0}{\theta}\,,\\
\at{\zeta+\eta}\, \at{r-\theta}+\zeta\,\theta&& \text{for}&r\in\cointerval{\theta}{\infty}\,,\\
\end{array}
\right.
\end{align}
and assume that the three free parameters $\zeta$, $\theta$, $\eta$ comply with the superlinearity condition
\begin{align}
\label{Eqn:Prms}
\zeta\geq0\,,\qquad \theta>0\,,\qquad \eta>0\,.  
\end{align}
\begin{figure}[t!]%
\centering{%
\includegraphics[width=0.8\textwidth]{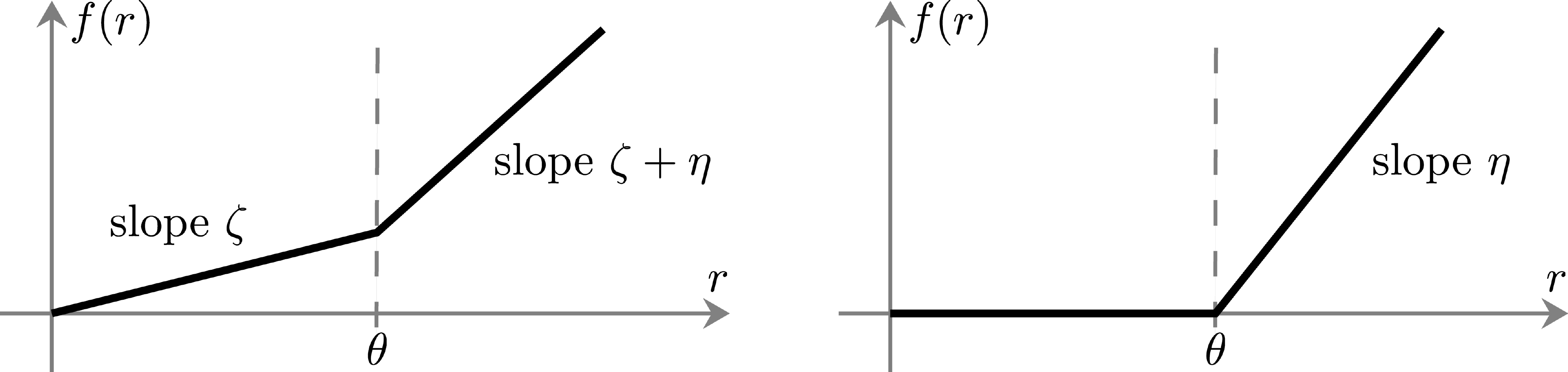}%
}%
\caption{%
The bilinear function $f$ from \eqref{Eqn:Nonl} as studied in this paper in the superlinear parameter regime \eqref{Eqn:Prms}. The general case (\emph{left panel}) can be traced back to the special case $\zeta=0$ (\emph{right panel}) according to Proposition~\ref{Prop:Transformation}.
}%
\label{Fig:Nonl}%
\end{figure}%
It is also essential for our uniqueness result that the kernel $a:\Rset\to\Rset$ is integrable and belongs to the cone
\begin{align*}
\calU = \big\{u\in\fspaceL^2\at\Rset
\;:\; u\at{x}=u\at{-x}\;\;\text{and}\;\; u\at{x}\geq u\at{\tilde{x}}\geq 0\;\;\text{for almost all}\;\; 0<x<\tilde{x}\big\}\,,
\end{align*}
which contains all bounded $\fspaceL^2$-functions that are  even, nonnegative, and unimodal. We suppose further that $a$ is differentiable and strictly unimodal according to
\begin{align}
\label{Eqn:Strict}
a^\prime\at{x}<0\quad \text{for}\quad x>0\,,
\end{align}
because this simplifies the presentation. However, we always discuss how more general results can be obtained by relaxing the precise formulation of theorems or enhancing the technical arguments in their proofs.
\begin{assumption}[properties of the convolution kernel]
\label{Ass:Kernel}%
The convolution kernel $a$ belongs to $\calU$ and $ \fspaceL^1\at{\Rset}$, is normalized by
\begin{align}
\notag
\int\limits_{-\infty}^{+\infty}a\at{x}\dint{x}=1\,,
\end{align}
and satisfies additionally \eqref{Eqn:Strict}.
\end{assumption}
A prototypical example for Assumption \ref{Ass:Kernel} is the Gaussian
\begin{align}
\label{Eqn:Kernel1}
a_1\at{x}:=\exp\at{-x^2}/\sqrt\pi\,
\end{align}
while the unimodal tent map
\begin{align}
\label{Eqn:Kernel2}
a_2\at{x}:=\max\{0,1-\abs{x}\}
\end{align}
and the indicator function
\begin{align}
\label{Eqn:Kernel3}
a_3\at{x}:=\left\{\begin{array}{ccl}1&& \text{for $\abs{x}\leq \tfrac12$}\\0&&\text{else} \end{array}\right.
\end{align}
are compactly supported and violate \eqref{Eqn:Strict}. Both kernels are nonetheless interesting since they are naturally related to traveling waves in FPUT chains and other spatially discrete conservation laws, see \eqref{Eqn:FPU} and \eqref{Eqn:DCL}.  For the kernel \eqref{Eqn:Kernel2}, the eigenvalue problem \eqref{Eqn:EV1} with bilinear function \eqref{Eqn:Nonl} has already been solved in \cite{TV14} by
a combination of Fourier and numerical methods, see the more detailed comments at the end of \S\ref{sect:non}.
\bigpar 
Our findings for the nonlinearity \eqref{Eqn:Nonl} and kernels as in Assumption \ref{Ass:Kernel} can be summarized as follows. %
\begin{figure}[t!]%
\centering{%
\includegraphics[width=0.98\textwidth]{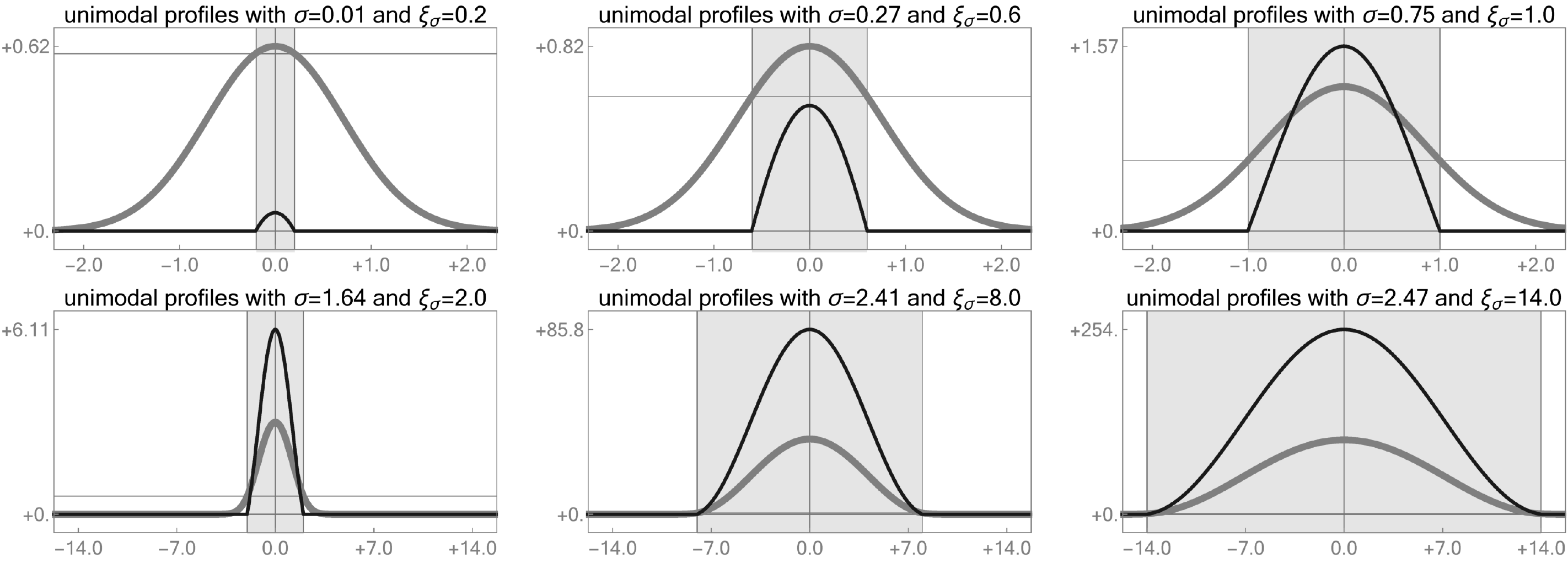}%
}%
\caption{%
Numerical profiles $u_\si$ (gray) and $f\at{u_\si}$ (black) for several values of $\si$,  the regular kernel \eqref{Eqn:Kernel1},
and nonlinearity parameters $\zeta=0$, $\theta=0.6$, $\eta=2.5$. The gray box represents the interval $I_{\xi_\si}$ which is defined by $u_\si\at{\pm\xi_\si}=\theta$. Notice the different plot regions for top and bottom row.
}%
\label{Fig.Ex1U}%
\bigskip%
\centering{%
\includegraphics[width=0.98\textwidth]{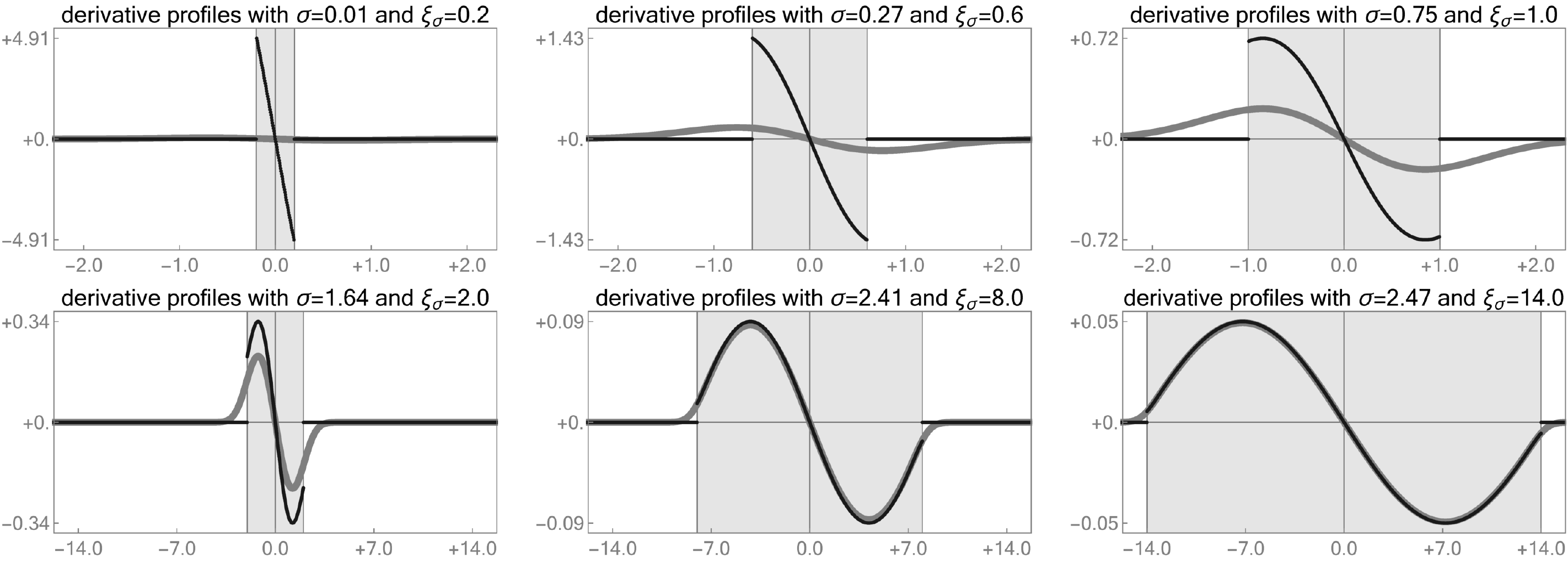}%
}%
\caption{%
Effective derivative profile $v_\si=\chi_{\xi_\si}\cdot u_\si^\prime$ (gray) as well as $a\ast v_\si=\eta^{-1}\,\si\, u_\si^\prime$ (black) for the nonlinear eigenfunctions from Figure \ref{Fig.Ex1U}. The Krein-Rutman eigenfunction provided by Proposition \ref{Prop:KR} for $\xi_\si$ is proportional to $v_\si$.
}%
\label{Fig.Ex1N}%
\end{figure} %
\begin{mresult}[existence and uniqueness under shape constraint]\quad %
\begin{enumerate}
\item
There exists a unique one-parameter family of nonlinear eigenfunction $u_\si\in\calU$ parametrized by $\zeta<\si<\zeta+\eta$.
\item
Formal asymptotic expansions suggest the following limit behavior:
\begin{enumerate}
\item%
The eigenfunctions $u_\si$ converge as $\si\searrow\zeta$ to a nontrivial limit profile $u_\zeta\in\calU$ with 
\begin{align*}
\qquad u_\zeta\at{x}\leq \theta
\quad\text{for all}\quad x\,,\qquad u_\zeta\at{0}=\theta\,,
\end{align*}
but the details are different for $\zeta=0$ and $\zeta>0$.
\item
There is no limit $u_{\zeta+\eta}$ due to $\lim_{\si\nearrow \zeta+\eta}\norm{u_\si}_2=\infty$.
\end{enumerate}
\end{enumerate}
\end{mresult}
The rigorous part of our main result is proven in Proposition \ref{Prop:ExUni} while the nonrigorous asymptotic analysis is presented in the appendix.  We also
refer to  Figures \ref{Fig.Ex1U}, \ref{Fig.Ex1N}, and \ref{Fig.Ex3} for numerical simulations with $\zeta=0$ and emphasize that our results do not exclude the existence of further eigenfunctions outside the cone $\calU$. In addition to the trivial nonuniqueness due to the shift invariance of \eqref{Eqn:EV1}, we expect the existence of families of periodic solutions. Moreover, for rapidly decaying kernels one might think about multi-bump solutions, which can be imagined as superposition of finitely many and well separated single-bump solutions having unimodal profile.
\begin{figure}[t!]%
\centering{%
\includegraphics[width=0.98\textwidth]{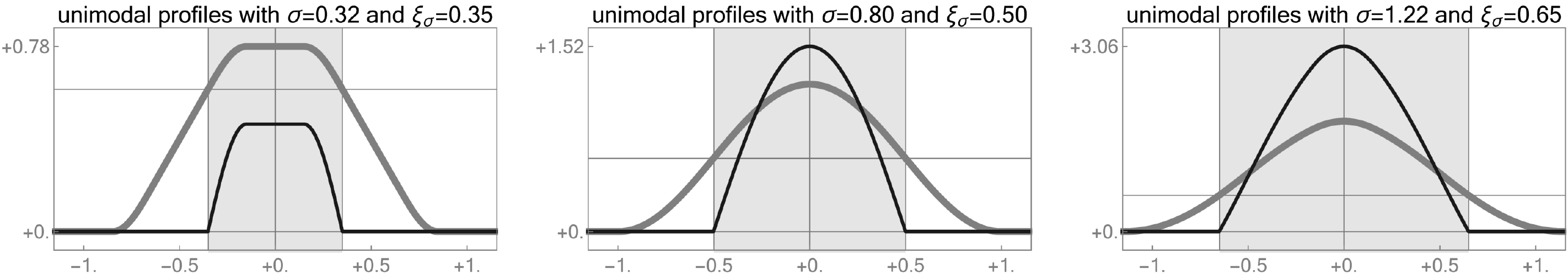}%
\\%
\includegraphics[width=0.98\textwidth]{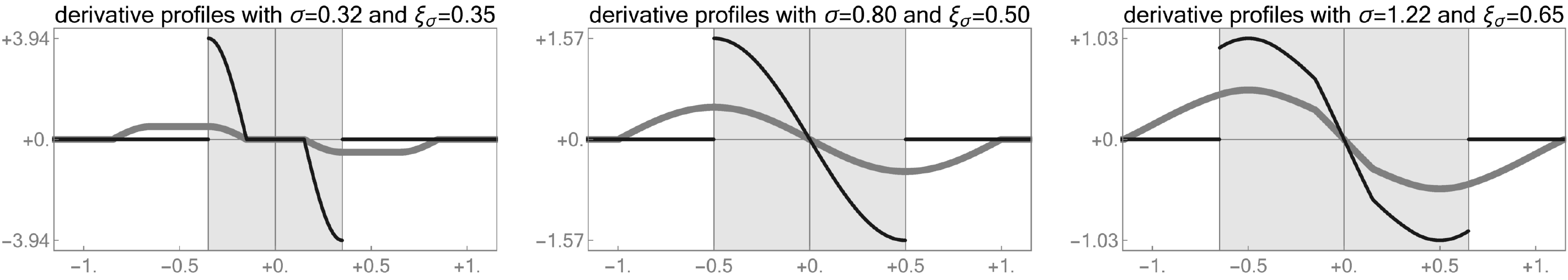}%
}%
\caption{%
Numerical simulations for the degenerate kernel 
\eqref{Fig.Ex3} and nonlinearity parameters as in Figure \ref{Fig.Ex1U}. The statement of some results must be modified in this case since the strict unimodality condition \eqref{Eqn:Strict} is not satisfied. 
}%
\label{Fig.Ex3}%
\end{figure}%
%
%
%
\paragraph{Proof strategy and organization of paper}
%
The central ideas for the proof of our main results can be sketched as follows. Assuming smoothness of $u$ we deduce $u^\prime\in\calN$, where
\begin{align*}
\calN:=\Big\{ v\in\fspaceL^2\at\Rset \,:\, v\at{x}=-v\at{-x}\;\;\text{and}\;\; v\at{x}\leq 0\;\;\text{for almost all}\;\; x>0 \Big\}\,,
\end{align*}
contains all odd functions that attain nonpositive values on $\oointerval{0}{\infty}$ (usually, elements of $-\calN$ and $\calN$ are called positive and negative, respectively). Moreover, differentiating \eqref{Eqn:EV1} we find
\begin{align}
\label{Eqn:NesCond1}
\si \,u^\prime = a \ast\bat{ \zeta u^\prime + \at{\eta-\zeta}\,\chi_\xi\cdot\, u^\prime} 
\end{align}
with
\begin{align}
\notag
 I_\xi := \ccinterval{-\xi}{\xi}\,,\qquad
 \chi_\xi\at{x}:=\left\{
\begin{array}{ccl}
1&&\text{for $x\in I_\xi$}\,,
\\%
0&&\text{else}\,,
\end{array}\right.\qquad
\end{align}
where $\xi$ is the value at which $u$ attains the value $\theta$. This reads
\begin{align}
\label{Eqn:NesCond2}
\theta = u\at{\xi}=\int\limits_{\xi}^\infty u^\prime\at{x}\dint{x}\,.
\end{align}
In the special case of $\zeta=0$, the necessary condition \eqref{Eqn:NesCond1} reads
\begin{align}
\label{Eqn:NesCond3a}
\la\, u^\prime = a \ast v
\end{align}
and involves the effective derivative profile
\begin{align}
\label{Eqn:NesCond3c}
v=\chi_\xi \cdot u^\prime
\end{align}
as well as
\begin{align}
\label{Eqn:NesCond3d}
\la = \si/\eta\,.
\end{align}
Formula \eqref{Eqn:NesCond3a} implies the linear eigenvalue problem
\begin{align}
\label{Eqn:NesCond3b}
 \la \,v = \calA_\xi v\,,
\end{align}
where
\begin{align}
\label{Eqn:DefOpA}
\calA_\xi v := \chi_\xi \cdot \,\bat{a\ast \at{\chi_\xi \cdot v}}
\end{align}
is a compact and symmetric operator, which exhibits nice invariance properties. In particular, adapting the strong version of the Krein-Rutmann theorem we can show that the eigenspace to the largest eigenvalue of $\calA_\xi$ is simple and generated by an element of $\calN$, see Proposition \ref{Prop:KR}.  This observation yields in combination with \eqref{Eqn:NesCond2} the existence and uniqueness of a family of solutions to \eqref{Eqn:EV1} which is naturally parametrized by $\xi$ and can be computed efficiently. Moreover, a closer inspection of the linear eigenvalue problem \eqref{Eqn:NesCond3b} as carried out in Proposition \ref{Prop:KRMon} reveals that the mapping $\xi\mapsto\la\mapsto \si$ is invertible so that $\xi$ can be replaced by $\si$. This change of parameters is the nonlinear part of our approach.
\par 
Another observation extends the existence and uniqueness result to the general case  $\zeta>0$. More precisely, we identify in Proposition \ref{Prop:Transformation} a modified kernel $\tilde{a}\in\calU$ with $\int_\Rset \tilde{a}\at{x}\dint{x}=1$ such that the implication
\begin{align}
\label{Eqn:LinTrafo}
 w-\mu\, a\ast w = a\ast g \qquad \Longleftrightarrow \qquad w=\at{1-\mu}^{-1}\,\tilde{a}\ast g
 \end{align}
holds with $\mu=\si^{-1}\zeta$. This enables us to transform the general case $\zeta>0$ into the special case $\zeta=0$ by replacing the kernel $a$ with $\tilde{a}$ and changing $\si$. A similar argument applies on the level of the derivatives because \eqref{Eqn:NesCond1} can be written as $\tilde{\la}\,u^\prime = \tilde{a}\ast\at{ \chi_\xi\cdot u^\prime}$. 
%
%
\section{Linear eigenvalue problem for the effective derivative profile}\label{sect:lin}
%
%
In this section we study the eigenvalue problem of the operator $\calA_\xi$ from \eqref{Eqn:DefOpA} in the space 
$\fspaceL^2_\odd\at\Rset$ which contains all functions that are square-integrable and odd. Notice that the representation formulas
 \begin{align}
\label{Eqn:OpA2}
\at{\calA_\xi v}\at{x}= \int\limits_{-\xi}^{+\xi}a\at{x-y}v\at{y}\dint{y}=\int\limits_0^{\xi}\bat{a\at{x-y}-a\at{x+y}}v\at{y}\dint{y}\qquad\text{for}\quad \abs{x}\leq \xi
\end{align}
and 
\begin{align*}
\at{\calA_\xi v}\at{x}= 0\qquad\text{for}\quad \abs{x}>\xi
\end{align*}
hold for any $v\in\fspaceL^2_\odd\at\Rset$.
\begin{lemma}[properties of the modified convolution operator]
\label{Lem:OpA}
For any $\xi\in\oointerval{0}{\infty}$, the operator $\calA_\xi$, maps $\fspaceL^2\at\Rset$ into itself, respects  the even-odd parity, and is both compact and self-adjoint. Moreover, the convex cones $\calN$ and $\calU$ are invariant under the action of $\calA_\xi$.
\end{lemma}
\begin{proof}
The first assertions can be derived from \eqref{Eqn:DefOpA} using standard arguments for convolution and multiplication operators. Moreover, since $a$ is even and unimodal,
the invariance of $\calN$ and $\calU$ follows from \eqref{Eqn:OpA2} and the analogous formula in $\fspaceL^2_\even\at\Rset$.
\end{proof}
%
%
\paragraph{Linear uniqueness result}
%
As first main auxiliary result we prove that the largest eigenvalue of $\calA_\xi$ is simple and spanned by a unique normalized eigenfunction in $\calN$. Such uniqueness results are usually inferred from the strong version of the Krein-Rutmann theorem but the classical formulation requires an order preserving operator that maps the cone of positive (or negative) elements in its interiors, see for instance \cite[appendix to chapter IIX]{DLvol3}. In our case, however, the cone $\calN$ has no inner points and the compact support of $\chi_\xi$ implies for every $v\in\calN$ that $\calA_\xi v$ belongs to the topological $\fspaceL^2$-boundary of $\calN$. 
\par
The assertions of the Krein-Rutmann Theorem hold notwithstanding. The crucial idea is that any eigenfunction that corresponds to the largest eigenvalue of $\calA_\xi$ must belong to the smaller cone
\begin{align*}
\widetilde{\calN}_\xi:=\Big\{ v\in \calN\cap \fspace{C}^1\at{I_\xi}\;:\;\,v^\prime\at{0}<0\;\;\text{and}\;\; v\at{x}<0\;\;\text{for}\;\;0<x\leq \xi\;\;\text{but}\;\;v\at{x}=0\;\;\text{for}\;\;x>\xi \Big\}\,,
\end{align*}
and this observation finally enables us to adapt classical arguments. A similar strategy has been used in \cite{BFH14}, which combines a variant of the operator $\calA_\xi$ with taylor-made nonlinear fixed point arguments to prove the existence (but not the uniqueness) of unimodal and compactly supported solutions $\varrho$ to the stationary variant of \eqref{Eqn:BAM2}. The involved linear Krein-Rutmann argument is that $v=\varrho^\prime$ satisfies the linearized equation $g^\prime\at\varrho\, v = a\ast v$ on the compact support of $\varrho$.
\begin{proposition}[variant of the Krein-Rutmann theorem]
\label{Prop:KR}
For any $0<\xi<\infty$, the largest eigenvalue $\la_\xi$ of $\calA_\xi$ in $\fspaceL_\odd^2\at\Rset$ is simple and the corresponding eigenspace is spanned by a unique eigenfunction $v_\xi\in\widetilde{\calN}_\xi$ with $\norm{v_\xi}_2=1$. 
\end{proposition}
\begin{proof} 
\emph{\ul{Variational setting}}\,: %
The operator $\calA_\xi$ admits only real eigenvalues thanks to Lemma \ref{Lem:OpA}. It is also the G\^ ateaux derivative of the functional
\begin{align*}
\calF_\xi\at{v}:=\tfrac12\bskp{v}{\calA_\xi  v}=
\tfrac12\bskp{\chi_\xi \cdot v}{a\ast \at{\chi_\xi \cdot v}}\,,
\end{align*}
which is well-defined on $\fspaceL^2_\odd\at\Rset$ since Young's inequality implies
\begin{align}
\label{Prop:KR.PEqn5}
2\babs{\calF_\xi\at{v}}\leq\bnorm{\chi_\xi\cdot v}_2\,\bnorm{a\ast \at{\chi_\xi\cdot v}}_2\leq \bnorm{a}_1\bnorm{\chi_\xi}_\infty^2\bnorm{ v}_2^2\leq \norm{v}_2^2\,.
\end{align}
Moreover, the largest eigenvalue $\la_\xi$ of $\calA_\xi$
can be characterized variationally via
\begin{align}
\label{Prop:KR.PEqn4}
\la_\xi = \max\big\{ 2\calF_\xi\at{v}\;:\;v\in\fspaceL^2_\odd\at\Rset\,,\; \norm{v}_2=1\big\}\,,
\end{align}
where any maximizer corresponds to a normalized eigenfunction to $\la_\xi$ and vice versa.
\par
\emph{\ul{Existence of eigenfunction in $\calN$}}\,: %
Any $v\in\fspaceL^2_\odd\at\Rset$ admits a unique and disjoint-support splitting 
\begin{align}
\label{Prop:KR.PEqn1}
v=v_--v_+\qquad\text{with}\qquad v_-,\,v_+\in\calN \qquad\text{and}\qquad\norm{v}_2^2=\norm{v_-}_2^2+\norm{v_+}_2^2\,,
\end{align}
and using Assumption \ref{Ass:Kernel}, Lemma \ref{Lem:OpA}, as well as \eqref{Eqn:DefOpA} we verify
\begin{align}
\label{Prop:KR.PEqn2}
\calF_\xi\at{v}=\tfrac12\bskp{v_-}{\calA_\xi v_-}+\tfrac 12 \bskp{v_+}{\calA_\xi v_+}-
\bskp{v_-}{\calA_\xi v_+}&\leq \calF_\xi\at{v_+}+\calF_\xi\at{v_-}\,.
\end{align}
Now suppose that $v\in\fspaceL^2_\odd\at\Rset$ is a  maximizer for \eqref{Prop:KR.PEqn4} with $v_-\neq0$ and $v_+\neq0$. We then have
\begin{align}
\label{Prop:KR.PEqn3}
\calF_\xi\at{v_{-}/\norm{v_-}_2}\leq \calF_\xi\at{v}\,,\qquad
\calF_\xi\at{v_{+}/\norm{v_+}_2}\leq \calF_\xi\at{v}
\end{align}
while the homogeneity of $\calF_\xi$ along with \eqref{Prop:KR.PEqn2} yields
\begin{align*}
\calF_\xi\at{v}\leq \norm{v_-}_2^2 
\calF_\xi\bat{v_{-}/\norm{v_-}_2}+\norm{v_+}_2^2 
\calF_\xi\bat{v_{+}/\norm{v_+}_2}\,.
\end{align*}
The combination of the two latter results with \eqref{Prop:KR.PEqn1} and the normalization condition $\norm{v}_2^2=1$ implies equality in both parts of \eqref{Prop:KR.PEqn3}. We conclude that $v_-$ and $v_+$ are eigenfunctions to the maximal eigenvalue $\la_\xi$ that belong to $\calN$.
\par
\emph{\ul{Regularity by refined cone analysis}}\,: %
We next show that any eigenfunction $v\in\calN$ to $\la_\xi$ belongs in fact to the smaller cone $\widetilde{\calN}_\xi$. The strict unimodality and the evenness of $a$ imply
\begin{align}
\label{Prop:KR.PEqn0}
a\at{x-y}-a\at{x+y}>0\qquad\text{for all $0<x,y<\xi$}\,,
\end{align}
so \eqref{Eqn:OpA2} guarantees $\bat{\calA_\xi v}\at{x}<0$ and hence $v\at{x}<0$ for all $0<x\leq \xi$. Moreover, we have
\begin{align*}
\at{\calA_\xi v}^\prime\at{x}= \int\limits_0^{\xi}\Bat{a^\prime\at{x-y}-a^\prime\at{x+y}}v\at{y}\dint{y}\,,\qquad
\at{\calA_\xi v}^\prime\at{0}= -2\int\limits_0^{\xi}a^\prime\at{y}v\at{y}\dint{y}<0
\end{align*}
and obtain $v\in\fspace{BC}^1\at{I_\xi}$ with $v^\prime\at{0}<0$. Here we used $\la_\xi>0$, which holds because \eqref{Eqn:OpA2} and \eqref{Prop:KR.PEqn0} provide $\calF_\xi\at{v}>0$ for any nontrivial $v\in\calN$.
\par
\emph{\ul{Uniqueness of normalized eigenfunctions}}\,: %
We first show uniqueness within $\calN$ as follows. Any two normalized eigenfunctions $v_1$ and $v_2$  to $\la_\xi$ lie in $\widetilde{\calN}_\xi$ and the properties of this cone imply the existence of a parameter $0<s<\infty$ such that $v_1-s\,v_2$ still belongs to $\calN$ but no longer to $\widetilde{\calN}_\xi$ (because of $v_1^\prime\at{0}=s\,v_2^\prime\at{0}$ or $v_1\at{x}=s\,v_2\at{x}$ for at least one $0<x\leq \xi$). This, however, means $s=1$ as well as $v_1=v_2$ because otherwise $v_1-s\,v_2$ would be nontrivial eigenfunction in $\calN$ that violates the above regularity result. Secondly, suppose there exists a nontrivial and normalized eigenfunction $v$ that belongs to neither $\calN$ or $-\calN$. Our results derived so far imply that $v_-/\norm{v_-}$ and $v_+/\norm{v_+}$ are identically maximizers of \eqref{Prop:KR.PEqn4} and hence a contradiction to the assumption on $v$. In summary, we have shown that the eigenspace to $\la_\xi$ in $\fspaceL^2_\odd\at\Rset$ is one-dimensional and spanned by a unique normalized function $v_\xi\in\widetilde{\calN}_\xi$.
\end{proof}
\begin{figure}[t!]%
\centering{%
\includegraphics[width=0.31\textwidth]{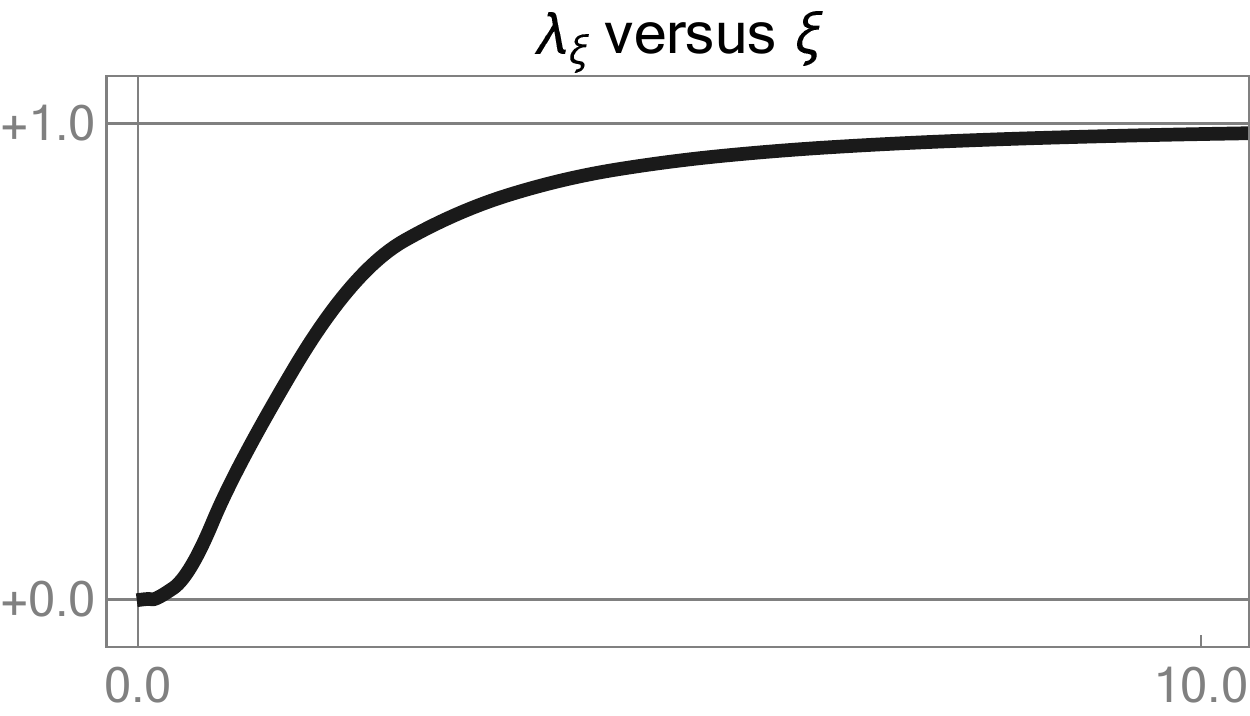}%
\hspace{0.025\textwidth}%
\includegraphics[width=0.31\textwidth]{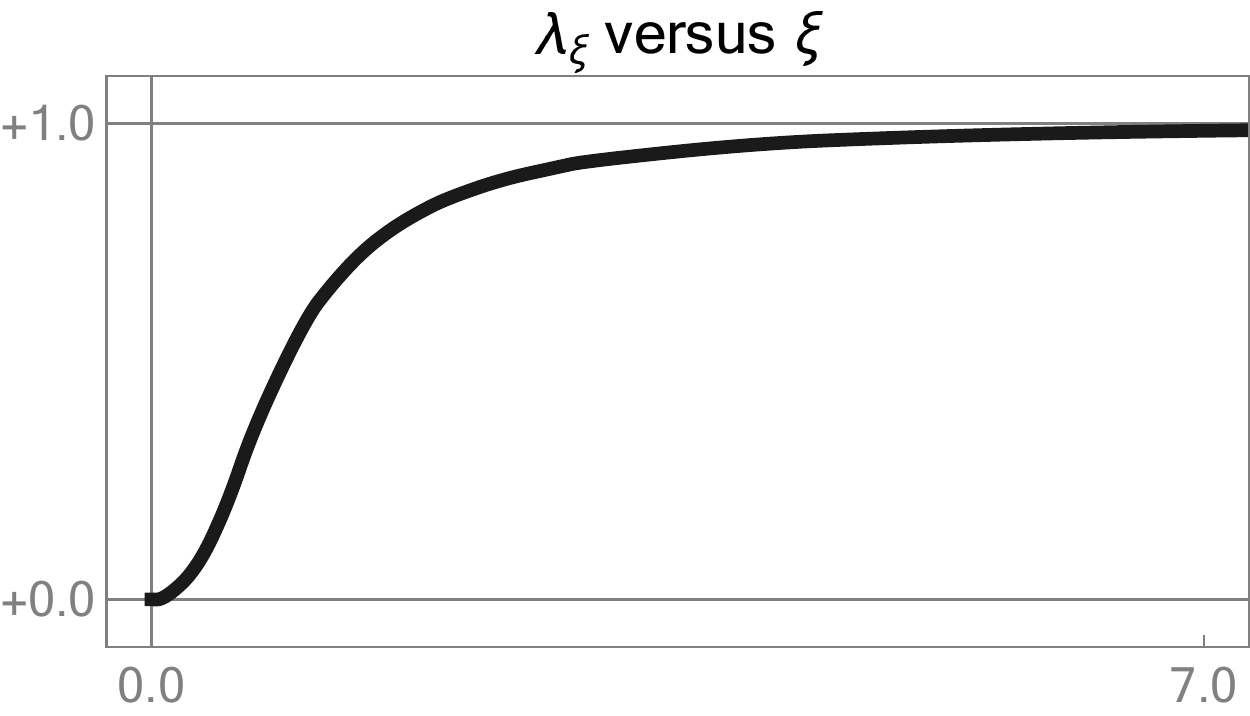}%
\hspace{0.025\textwidth}%
\includegraphics[width=0.31\textwidth]{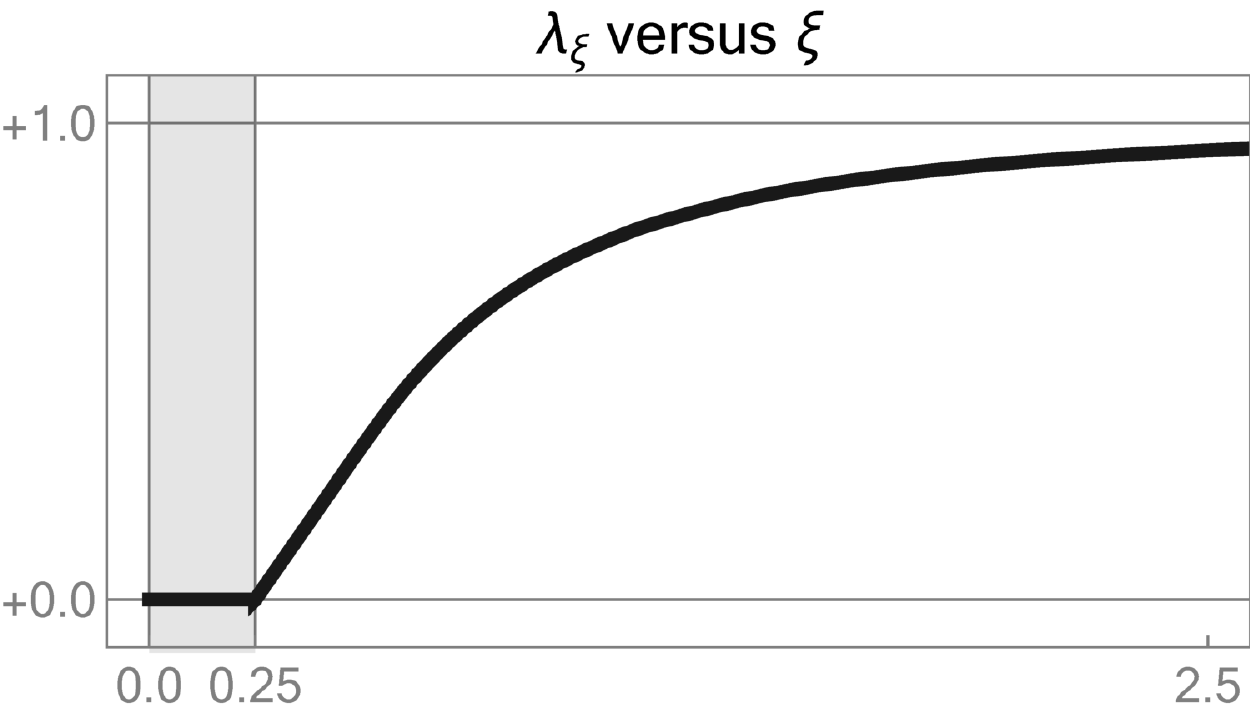}
}%
\caption{%
Numerical values of the Krein-Rutmann eigenvalue $\la_\xi$ from Proposition \ref{Prop:KRMon} for the kernels \eqref{Eqn:Kernel1} (\emph{left}), \eqref{Eqn:Kernel2} (\emph{center}), and \eqref{Eqn:Kernel3} (\emph{right}).  Notice the constant region in the last example (gray box), which stems from the plateau in the kernel.
}%
\label{Fig.KR}
\end{figure}%
%
%
%
\paragraph{Parameter dependence}
%
The second building block for our nonlinar uniqueness result in \S\ref{sect:non} are the following properties of the Krein-Rutmann quantities. Related numerical simulations are presented in Figure \ref{Fig.KR}.
\begin{proposition}[$\xi$-dependence of $\la_\xi$ and $v_\xi$]
\label{Prop:KRMon}
The map $\xi\mapsto\la_\xi$ from Proposition \ref{Prop:KR} is continuous and strictly increasing with
\begin{align*}
\lim_{\xi\searrow 0} \la_\xi=0\,,\qquad 
\lim_{\xi\nearrow \infty}\la_\xi =1
\end{align*}
Moreover, the corresponding map $ \xi\mapsto v_\xi$ is $\fspaceL^2$-continuous.
\end{proposition}
\begin{proof}
\emph{\ul{Monotonicity}}\,: %
Let $0<\xi_1<\xi_2<\infty$ be fixed. Since $v_{\xi_1}$ from Proposition \ref{Prop:KR} is supported in $I_{\xi_1}$, we have
\begin{align*}
\chi_{\xi_1}\cdot  v_{\xi_1} = v_{\xi_1} = \chi_{\xi_2}\cdot  v_{\xi_1}\,,\qquad 
\calF_{\xi_1}\at{v_{\xi_1}}=
\calF_{\xi_2}\at{v_{\xi_1}}\,,
\end{align*}
and the optimization problem \eqref{Prop:KR.PEqn4} ensures the nonstrict monotonicity $\la_{\xi_1}\leq\la_{\xi_2}$. However, $v_{\xi_1}$ cannot be an eigenfunction of $\calA_{\xi_2}$ because the properties of $a$ imply that $\calA_{\xi_2} v_{\xi_1}$ has a larger support than $v_{\xi_1}\in\widetilde{\calN}_{\xi_1}$. We thus obtain $\la_{\xi_1}\neq\la_{\xi_2}$. 
\par
\emph{\ul{Continuity and asymptotics for small $\xi$}}\,:
Suppose that $\at{\xi_n}_{n\in\Nset}$ is a given sequence with
\begin{align*}
\xi_n\to\xi\in\oointerval{0}{\infty}\,, \qquad \la_{\xi_n}\to\la\in\oointerval{0}{\infty}\,.
\end{align*} 
By weak compactness we can extract a (not relabeled) subsequence such that the normalized eigenfunctions $v_{\xi_n}$ converge weakly to some limit $v\in\fspaceL^2_\odd\at\Rset$, where the weak closedness of $\calN$ ensures $v\in\calN$. The properties of convolution operators combined with $\chi_{\xi_n}\to\chi_n$ imply that $\calA_{\xi_n} v_{\xi_n}$ converges strongly to $\calA_{\xi}v$ and this guarantees --- thanks to the eigenvalue equation for $v_{\xi_n}$ --- the strong convergence of $v_{\xi_n}$. In particular, we obtain
\begin{align}
\label{Prop:KRMon.PEqn1}
\la\, v = \calA_\xi v\,,\qquad \norm{v}_2=1\,,\qquad 2\calF_\xi\at{v}=\la
\end{align}
and hence $\la\leq\la_\xi$. On the other hand, \eqref{Prop:KR.PEqn4} ensures $\la_\xi\leq \la$ since
\begin{align*}
2\calF_\xi\at{\tilde{v}}=\lim_{n\to\infty} 2\calF_{\xi_n}\at{\tilde{v}}\leq\lim_{n\to\infty}\la_{\xi_n}=\la
\end{align*}
holds for any $\tilde{v}\in\fspaceL^2_\odd\at\Rset$. We thus get $\la=\la_\xi$ and Proposition \ref{Prop:KR} combined with \eqref{Prop:KRMon.PEqn1} yields $v=v_{\xi}$. In summary, we have shown for any convergent sequence $\at{\xi_n}_{n\in\Nset}$ that $\la_\xi$ is the unique accumulation point of the sequence $\bat{\la_{\xi_n}}_{n\in\Nset}$ (which is bounded due to the monotonicity with respect to $\xi$), and this gives rise to the claimed continuity of $\la_\xi$. Our arguments also yields the strong convergence $v_{\xi_n}\to v_\xi$ and analogously we derive $\xi_n\to0$ from $\la_{\xi_n}\to0$ because $\calA_{\xi_n}v_{\xi_n}\to0$ holds along any weakly convergent subsequence.
\par
\emph{\ul{Asymptotics for large $\xi$}}\,:
Direct computations for the piecewise constant and compactly supported test function $\tilde{v}_\xi $ with
\begin{align*}
\tilde{v}_\xi\at{x} = -\at{2\xi}^{-1/2}\sgn\at {x}\,\chi_\xi\at{x}
\end{align*} 
yield $\norm{\tilde{v}_\xi}_2=1$ as well as
\begin{align*}
\calF_\xi\at{\tilde{v}_\xi}&=\frac{1}{4\xi}\int\limits_{-\xi}^{+\xi}\int\limits_{-\xi}^{+\xi}a\at{x-y}\,\sgn\at{x}\,\sgn\at{y}\dint{y}\dint{x}=
\frac{1}{2\xi}\int\limits_{0}^{\xi}\int\limits_{0}^{\xi}\bat{a\at{x-y}-a\at{x+y}}\dint{y}\dint{x}
\\&= \frac{1}{2\xi}\int\limits_{0}^{2\xi}\int\limits_{0}^{+t} \bat{a\at{s}-a\at{t}}\dint{s}\dint{t}
=\int\limits_0^{2\xi} a\at{z}\frac{2\xi-2z}{2\xi}\dint{z}\,.
\end{align*}
Using \eqref{Prop:KR.PEqn4} and the Dominated Convergence Theorem we therefore get
\begin{align*}
\liminf_{\xi\to\infty} \la_\xi\geq 2\lim_{\xi\to\infty} \calF_\xi\at{\tilde{v}_\xi}=2\int\limits_0^\infty a\at{z}\dint{z}=1
\end{align*}
thanks to Assumption \ref{Ass:Kernel}. On  the other hand, the Young estimate \eqref{Prop:KR.PEqn5} applied to $v_\xi$ ensures
\begin{align*}
\limsup_{\xi\to\infty} \la_{\xi}=2\limsup_{\xi\to\infty}\calF_\xi\at{v_\xi}\leq1\,,
\end{align*} 
and  the proof is complete.
\end{proof}
%
%
\paragraph{Approximation}
%
%
The unique normalized Krein-Rutmann eigenfunction as provided by Proposition \ref{Prop:KRMon} can be computed as limit of the sequence $\at{v_n}_{n\in\Nset}\subset\calN$ with 
\begin{align*}
v_n:=\frac{\calA_\xi ^nv_0}{\bnorm{\calA_\xi ^nv_0}_2}
\end{align*}
and arbitrary $v_0\in\calN$. Moreover, using $\calF_\xi\at{-\sgn \abs{v}}\geq \calF_\xi\at{v}$ we deduce that $\la_\xi$ is not only the largest eigenvalue but also the spectral radius of $\calA_\xi$ and hence the exponential convergence of $v_n$ as $n\to\infty$. This approximation scheme is often called Power Method and a straight forward discretization (fine but equidistant spatial grid and Riemann sums instead of integrals) has been used to produce the numerical results in this paper. Alternatively, one can write the sequence as
\begin{align*}
\la_n v_{n+1}=\calA_\xi v_n\,,\qquad \la_n=\norm{\calA_\xi v_n}_2\,,
\end{align*}
where $\la_n$ will converge as $n\to\infty$ to the Krein-Rutmann eigenvalue $\la_\xi$. A similar \emph{improvement dynamics} can also be applied to the nonlinear eigenvalue problem \eqref{Eqn:EV1}, but the convergence is more subtle due to the lack of uniquness results for general functions $f$. We refer to \cite{HM20} for numerical examples and a more detailed discussion of the analytical properties.
%
\paragraph{Generalizations}
%
%
We finally discuss the case of less regular kernels $a$.
The crucial ingredient to the regularity step $v\in\calN \Rightarrow  \calA_\xi v\in\widetilde{\calN}_\xi$ in the proof of Proposition \ref{Prop:KR} is \eqref{Prop:KR.PEqn0}. This condition hinges on the strict unimodality of $a$ as in  \eqref{Eqn:Strict} and is not satisfied for all $a\in\calU$, see Figure \ref{Fig.Supports} for an illustration.
For kernels like the tent map \eqref{Eqn:Kernel2}, a recursive argument reveals the implication $v\in\calN \Rightarrow \calA_\xi^n  v\in \widetilde{\calN}_\xi$ for all sufficiently large $n\in\Nset$, and this guarantees that all assertions in Proposition \ref{Prop:KR} remain valid since the Krein-Rutmann eigenfunction still belongs to $\tilde{\calN}_\xi$. This, however, is no longer true for the kernel \eqref{Eqn:Kernel3}. Instead, we have to distinghuish between the following three parameter regimes: 
\begin{enumerate}
\item %
$0<\xi\leq\tfrac14$ implies that  $\calA_\xi v$ vanishes for any $v\in\fspaceL^2_\odd\at\Rset$, so $\calA_\xi$ restricted to $\fspaceL^2_\odd\at\Rset$ is actually the trivial operator and we have $\la_\xi=0$, see the third panel in Figure \ref{Fig.KR}. 
\item 
In the case of $\tfrac14<\xi<\tfrac12$, there exists a unique and normalized Krein-Rutmann eigenfunction $v_\xi$, which does not belong to $\widetilde{\calN}_\xi$ but to a modified subcone of $\calN$ with
\begin{align*}
v^\prime\at{\tfrac12-\xi}<0\,,\qquad
v\at{x}<0\;\; \text{for}\;\;\tfrac12-\xi<x\leq \xi\,,\qquad
v\at{x}=0\;\;\text{for}\;\; 0<x<\tfrac12-\xi \;\text{and}\; x>\xi\,,
\end{align*}
see the first column in Figure \ref{Fig.Ex3}.
\item For $\xi\geq\tfrac12$, we find again $v_\xi\in\widetilde{\calN}_\xi$.
\end{enumerate}
In summary, if $a$ is given by \eqref{Eqn:Kernel3}, then 
Proposition \ref{Prop:KR} must be formulated more carefully
and the analogue to Proposition \ref{Prop:KRMon} stipulates the one-to-one correspondence between $\xi\in\oointerval{\tfrac14}{\infty}$ and $\la_\xi\in\oointerval{0}{1}$. Similar statements apply to any other kernel $a\in\calU$ that admits a constant plateau near $x=0$.
\begin{figure}[t!]%
\centering{%
\includegraphics[width=0.5\textwidth]{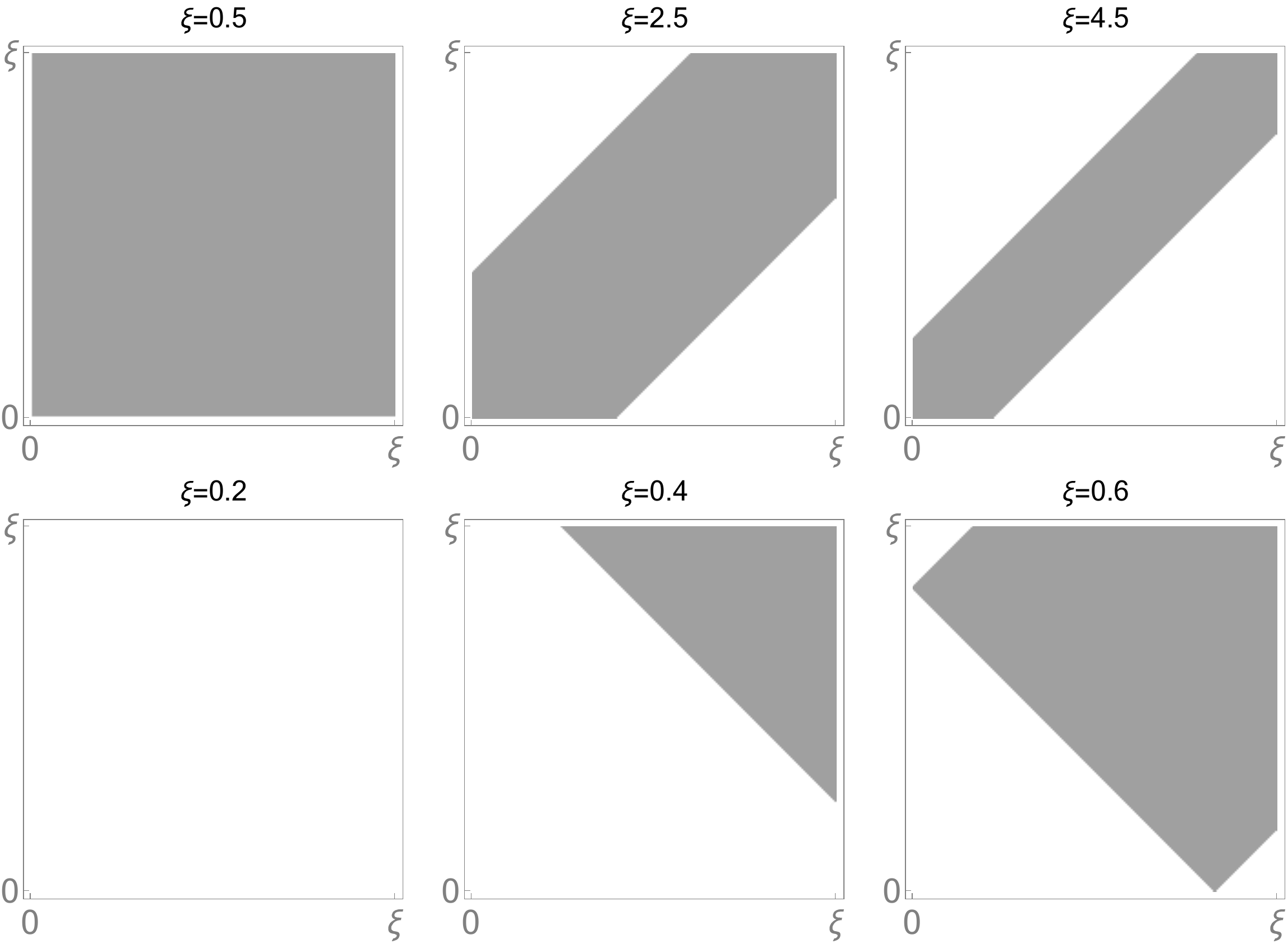}}%
\caption{%
Support of the integral kernel corresponding to $\calA_\xi$, see \eqref{Eqn:OpA2}, for several values of $\xi$ and the degenerate convolution kernels  \eqref{Eqn:Kernel2} (\emph{top row}) and \eqref{Eqn:Kernel3} (\emph{bottom row}). Since condition \eqref{Prop:KR.PEqn0} is not satisfied for some values of $\xi$, the proof and/or the precise statement of Proposition \ref{Prop:KR} needs to be modified as discussed in the text. See also Figure \ref{Fig.KR} for the corresponding Krein-Rutmann eigenvalues.
}%
\label{Fig.Supports}%
\end{figure} %
%
%
%
\section{Nonlinear eigenfunctions with unimodal profile}\label{sect:non}
%
%
We are now able to prove our main result from \S\ref{sect:intro} in two steps. First we show that the general case with given kernel $a$ and $\zeta>0$ is equivalent to the special case $\zeta=0$ for a modified kernel $\tilde{a}$.
\begin{proposition}[transformation to special case]
\label{Prop:Transformation}
A function $u$ solves \eqref{Eqn:EV1} with $\si\in \oointerval{\zeta}{\zeta+\eta}$ if and only
the eigenvalue equation
\begin{align}
\label{Eqn:EV2}
\tilde{\si}\,u=\tilde{a}\ast \tilde{f}\at{u}
\end{align}
is satisfied with 
\begin{align*}
\tilde{\si}:=\si-\zeta=\at{1-\mu}\,\sigma\in\oointerval{0}{\eta}\,,\qquad 
\mu := \sigma^{-1}\zeta \in \oointerval{0}{1}\,.
\end{align*}
Here, the transformed kernel
\begin{align}
\label{Prop:Transformation.Eqn1}
\tilde{a}:=
\at{1-\mu}\,a\ast\bat{1+\mu\, a + \mu^2\, a\ast a +
\mu^3 a\ast a\ast a +\hdots}\,,
\end{align}
is well-defined, depends on $\si$, and complies with Assumption \ref{Ass:Kernel}, while the simplified nonlinearity
\begin{align*}
\tilde{f}\at{r}:=\left\{
\begin{array}{lccl}
0&& \text{for}& r\in\ccinterval{0}{\theta}\,,\\
\eta\, \at{r-\theta}&& \text{for}&r\in\cointerval{\theta}{\infty}\,,\\
\end{array}
\right.
\end{align*}
represents \eqref{Eqn:Nonl} with $\zeta=0$.
\end{proposition}
\begin{proof}
\emph{\ul{Linear auxiliary operator and modified kernel}}\,: %
The Young estimate
\begin{align*}
\norm{a\ast w}_2\leq \norm{a}_1\norm{w}_2=\norm{w}_2
\end{align*}
combined with $0<\mu<1$ implies that the linear operator 
\begin{align*}
\calL:\fspaceL^2\at\Rset\to\fspaceL^2\at\Rset\,,\qquad \calL  w := w-\mu\,a \ast w
\end{align*}
is continuously invertible. The Neumann formula 
\begin{align*}
\calL^{-1}w = \breve{a}\ast w\,,\qquad \breve{a}:=1+\mu\,
a +\mu^2\,a\ast a +\,\mu^3\, a\ast a\ast a\ast \hdots
\end{align*} 
reveals $\bar{a}\in\calU$ thanks to Lemma \ref{Lem:OpA}, which in turn ensures that $\calL$ respects the even-odd parity and mediates the implication \eqref{Eqn:LinTrafo}. We further have
\begin{align*}
\tilde{a}=\at{1-\mu} \,a\ast \breve{a} = \at{1-\mu} \, \breve{a}\ast a \in \calU
\end{align*} 
and compute
\begin{align*}
\int\limits_\Rset\tilde{a}\at{x}\dint{x}=\frac{1}{1-\mu}\sum_{k=0}^\infty \mu^k=1
\end{align*}
as well as
\begin{align*}
\tilde{a}^\prime:=
\at{1-\mu}\,\breve{a}\ast a^\prime\,.
\end{align*}
The latter formula provides via
\begin{align}
\notag
\tilde{a}^\prime\at{x}<0\qquad\text{ for all $x>0$}
\end{align}
the strict unimodality of $\tilde{a}$.
\par
\emph{\ul{Transformation of the nonlinear problem}}\,: 
Let $\pair{\si}{u}$ be a given solution to \eqref{Eqn:EV1}. Our definitions imply 
\begin{align*}
\tilde{\sigma}\bat{u-\mu\, a \ast u}=\at{1-\mu }\bat{\si\, u - \zeta\, a\ast u} = \at{1-\mu }\bat{ a\ast \bat{f\at{u}-\zeta\, u}} = \at{1-\mu }\, a\ast \tilde{f}\at{u}
\end{align*} 
and hence the validity of \eqref{Eqn:EV2} thanks to the existence of $\calL^{-1}$. Similarly, the reverse implication follows by applying $\calL$ to both sides of \eqref{Eqn:EV2}.
\end{proof}
In the second step we finally establish our existence and uniqueness result for nonlinear eigenfunctions.
\begin{proposition}[nonlinear existence and uniqueness result]
\label{Prop:ExUni}
Equation \eqref{Eqn:EV1} admits for any $\si\in\oointerval{\zeta}{\zeta+\eta}$ a unique solution $u_\si\in\calU$. 
\end{proposition}
\begin{proof} 
\emph{\ul{Preliminaries}}\,: %
Within this proof, $\si$ is fixed. In view of Proposition \ref{Prop:Transformation}, we can assume $\theta=0$ with $\si\in\oointerval{0}{\eta}$ because otherwise we replace $a$ by $\tilde{a}$ and $\si$ by $\tilde{\si}$.  Moreover, Proposition \ref{Prop:KRMon} guarantees that 
\begin{align}
\label{Prop:ExUni.PEqn8}
\la_{\xi_\si}=\eta^{-1}\si
\end{align} 
holds for precisely one cut-off parameter $\xi_\si\in\oointerval{0}{\infty}$.
\par
\emph{\ul{Existence and construction}}\,: %
By integration, there exists a unique function $\tilde{u}_\si\in\calU$ with
\begin{align}
\label{Prop:ExUni.PEqn2}
\tilde{u}_\si^\prime\at{x}=v_{\xi_\si}\at{x}\quad\text{for}\quad x\in\Rset\,,\qquad \tilde{u}_\si\at{x}=0\quad\text{for}\quad \abs{x}\geq \xi_\si,
\end{align}
where $v_{\xi_\si}\in\widetilde{\calN}_\xi$ is the normalized and compactly supported Krein-Rutmann eigenfunction from Proposition \ref{Prop:KR}. Since $\tilde{u}_\si$ is strictly positive for $\abs{x}<\xi_\si$, the function $a\ast\tilde{u}_\si$ attains a positive value at $x=\xi_\si$. We thus define $u_\si\in\calU$ by
\begin{align}
\label{Prop:ExUni.PEqn1}
u_\si:=\tau_\si\,a\ast \tilde{u}_\si
\,,\qquad%
\tau_\si:=
  \frac{\theta}{\at{a\ast \tilde{u}_\si}\at\xi}
\end{align}
and observe that this guarantees
\begin{align*}
u_\si\at{\pm\xi_\si}=\theta\,,\qquad\quad \theta\leq u_\si\at{x}\quad \text{for}\quad \abs{x}<\xi_\si
\,,\qquad\quad 0\leq u_\si\at{x}\leq \theta\quad \text{for}\quad \abs{x}>\xi_\si
\end{align*} 
due to the unimodality of $u_\si$. Since we also have
\begin{align*}
u^\prime_\si\at{x} = \tau_\si\,\bat{a\ast \tilde{u}_\si^\prime}\at{x}=
\tau_\si\,\bat{a\ast v_{\xi_\si}}\at{x}=\tau_\si\,\la_{\xi_\si}\, v_{\xi_\si}\at{x}=\tau_\si\,\la_{\xi_\si}\tilde{u}^\prime_\si\at{x}
\end{align*}
for all $\abs{x}\leq{\xi_\si}$ (but not for $\abs{x}>\xi_\si$), we find
\begin{align*}
u_\si\at{x}=\theta+\tau_\si\,\la_{\xi_\si}\,\tilde{u}_\si\at{x}
\quad \text{for}\quad \abs{x}\leq \xi_\si\,.
\end{align*}
In particular, we have
\begin{align*}
f\at{u_\si\at{x}}=\tau_\si \,\la_{\xi_\si}\,\eta\,\tilde{u}_\si\at{x}\qquad \text{for all}\quad x\in\Rset
\end{align*}
thanks to \eqref{Eqn:Nonl} with $\zeta=0$, so
\begin{align*}
a\ast f\at{u_\si}=a\ast\bat{\tau_\si\la_{\xi_\si}\,\eta\,\tilde{u}_\si}= \la_{\xi_\si}\,\eta\, u_\si=\si\,u_\si
\end{align*}
follows from our definition of $u_\si$ in \eqref{Prop:ExUni.PEqn1}. 
\par%
\emph{\ul{Uniqueness}}\,: %
Now suppose that $0\neq u\in\calU$ solves \eqref{Eqn:EV1} for $\zeta=0$. This gives
$0<u\at{0}=\norm{u}_\infty$ (otherwise $f\at{u}$ would vanish identically) and there exists $\xi\in\oointerval{0}{\infty}$ with $u\at{\xi}=\theta$. By differentiating \eqref{Eqn:EV1} with respect to $x$ --- and using the notations  \eqref{Eqn:NesCond3c}, \eqref{Eqn:NesCond3d}  --- we establish the formulas \eqref{Eqn:NesCond3a} and \eqref{Eqn:NesCond3b}, so Propositions \ref{Prop:KR} and \ref{Prop:KRMon} provide
\begin{align*}
\xi=\xi_\si \qquad \text{and}\qquad v= c\, v_{\xi_\si}
\end{align*}
for some factor $c>0$. Combining this with \eqref{Prop:ExUni.PEqn2} and \eqref{Prop:ExUni.PEqn1} we conclude that the derivatives of $u$ and $u_\si$ are proportional, and the consistency relations
\begin{align*}
\int\limits_{\xi_\si}^\infty u^\prime\at{x}\dint{x}=u\at{\xi_\si} =\theta 
=u_\si\at{\xi_\si}=\int\limits_{\xi_\si}^\infty u_\si^\prime\at{x}\dint{x}
\end{align*}
give $u=u_\si$.
\end{proof}
%
%
%
%
\paragraph{Comments}
%
%
\begin{figure}[t!]%
\centering{%
\includegraphics[width=0.98\textwidth]{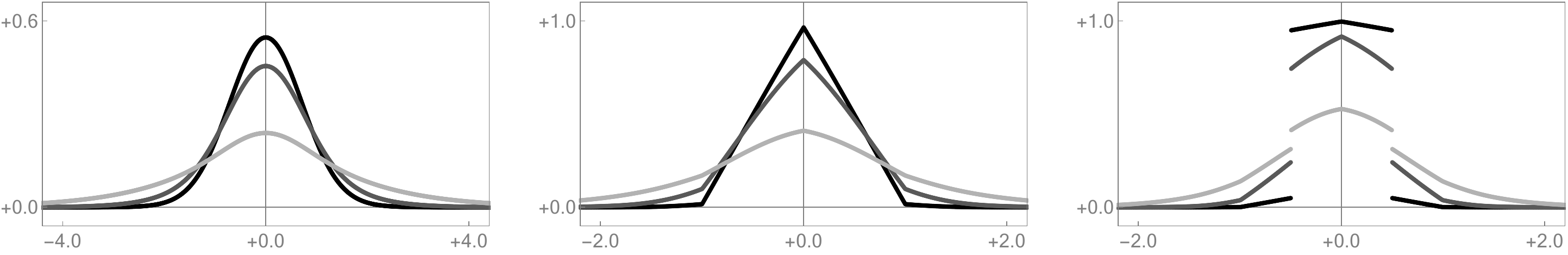}
}%
\caption{%
The transformed kernel \eqref{Prop:Transformation.Eqn1} for the Gaussian (\emph{left}), the tent map (\emph{middle}), and the indicator function (\emph{right}) from \eqref{Eqn:Kernel1}, \eqref{Eqn:Kernel2}, and \eqref{Eqn:Kernel3}. The 
three curves correspond to $\mu=0.1$ (black), $\mu=0.5$ (dark gray), $\mu=0.9$ (light gray).
}%
\label{Fig.kernels}%
\end{figure}%
Notice that the transformed kernel $\tilde{a}$ in \eqref{Prop:Transformation.Eqn1} depends also on $\zeta$ and satisfies the analogue to the crucial condition \eqref{Prop:KR.PEqn0} even if the kernel $a$ is not strictly unimodal due to a compact support or constant plateaus, see Figure \ref{Fig.kernels} for an illustration. Consequently, all results in this section cover for $\zeta>0$ the kernels \eqref{Eqn:Kernel2} and \eqref{Eqn:Kernel3} as well, and our comments at the end of \S\ref{sect:lin} on how to generalize Propositions \ref{Prop:KR} and \ref{Prop:KRMon} are relevant for $\zeta=0$ only.
\par
The transformed kernel $\tilde{a}$ also features prominently in \cite{TV14}, which studies the bilinear eigenvalue problem for the tent map kernel by different methods. In our notations, the main ideas can be described as follows. The function $\tilde{w}$ with
\begin{align*}
\tilde{w}\at{x}=\frac{\tilde{f}\bat{u\at{x}}}{\tilde{\si}}=\frac{\eta\,\chi_\xi\at{x}\, \bat{u\at{x}-\theta}}{\si-\zeta}\,,
\end{align*}
represents the inhomogeneity, is continuous (especially  at $x=\pm\xi$), and satisfies 
\begin{align*}
w^\prime\at{x}=\frac{\eta\,\chi_\xi\at{x}\,u^\prime\at{x}}{\si-\zeta}=\frac{\eta\,v\at{x}}{\si-\zeta}
\end{align*} 
for all $x\in\Rset$.  The linear problem for the effective derivative profile \eqref{Eqn:NesCond3c} can hence be written analogously to \eqref{Eqn:NesCond3b} as 
\begin{align}
\label{Eqn:DiffAppproach2}
v\at{x}-\frac{\si-\zeta}{\eta}\int\limits_{-\xi}^{+\xi}\tilde{a}\at{x-y} v\at{y}\dint{y}=0\qquad\text{for}\quad -\xi\leq x\leq+\xi
\end{align}
and \cite{TV14} solves this equation under the consistency relation $\int_0^\xi v\at{x}\dint{x}=\theta$ by combining two ingredients. The fixed point problem $\eqref{Eqn:DiffAppproach2}$ is solved with parameter $\xi$ by Wiener-Hopf methods. This requires detailed information on the Fourier transform of $\tilde{a}$ and direct computations reveal
\begin{align}
\notag
\widehat{\tilde{a}}\at{k} = \frac{\si\,\widehat{a}\at{k}}{\si-\zeta\, \widehat{a}\at{k}}=\frac{\si}{\eta}\at{1-\frac{\at{\zeta+\eta}\,\widehat{a}\at{k}\,k^2-\si\,k^2}{\zeta\,\widehat{a}\at{k}\,k^2-\si\,k^2}}
\end{align}
as well as $\widehat{a}\at{k}\,k^2=4\sin^2\at{k/2}$ for the kernel \eqref{Eqn:Kernel2}. Moreover, the value of $\xi$ is determined numerically by means of a scalar nonlinear equation equivalent to \eqref{Prop:ExUni.PEqn8}. This approach works well since $\widehat{a}$ is a nice known function and provides power series expressions for $v$ and $u$, which can be used to derive intricate but almost explicit approximation formulas.
\par
We finally mention that nonnegative eigenfunctions (being unimodal or not) cannot exist for $\si\geq\zeta+\eta$ since \eqref{Eqn:EV1} implies
\begin{align*}
\si\bnorm{u}_2\leq \bnorm{a\ast f\at{u}}_2\leq \bnorm{a}_1\,\bnorm{f\at{u}}_2\leq \at{\zeta+\eta}\bnorm{u}_2\,,
\end{align*}
where the last estimate is actually strict for any nontrivial $u$. Similarly, using
\begin{align*}
\si\int\limits_\Rset u\at{x}\dint{x}=\int\limits_\Rset f\bat{u\at{x}}\dint{x}
\end{align*}
we can disprove the existence of nonnegative eigenfunctions for $\si\leq\zeta$ under the additional assumption $u\in\fspaceL^1\at\Rset$. Notice that the eigenfunction from Proposition \ref{Prop:ExUni} are integrable due to \eqref{Eqn:EV2} and because $\tilde{f}\at{u_\si}$ is always compactly supported. More precisely,  
$u_\si$ decays as fast as $a$ or the transformed kernel $\tilde{a}$ for $\zeta=0$ and $\zeta>0$, respectively. A similar argument has been applied in \cite{HM20}, which proves the existence and the localization (but not the uniquness) of nonlinear eigenfunctions for more general functions $f$ in a nonlinear variational setting.
%
%
\appendix
%
\section{Formal asymptotic analysis}
%
%
We characterize the limiting behavior of the nonlinear eigenfunction $u_\si$ from Proposition \ref{Prop:ExUni}. We always start with the special case $\zeta=0$ and discuss the necessary modification in the case $\zeta>0$ afterwards. The natural quantity for the formal asymptotic analysis with $\zeta=0$ are the compactly supported functions
\begin{align}
\label{Eqn:App6}
w_\si\at{x} := \si^{-1}\,f\at{u_\si\at{x}} \,,
\end{align}
which determine $u_\si$ via
\begin{align}
\label{Eqn:App3}
u_\si = a\ast w_\si\,.
\end{align}
We rescale $w_\si$ according to
\begin{align*}
\ol{w}_\si\at{\ol{x}}=w\at{x}\,,\qquad x=\xi_\si\,\ol{x}\,,
\end{align*}
because this allows us to work on the fixed interval $\ol{I}:=\ccinterval{-1}{+1}$, and notice that
\begin{align*}
\ol{w}_\si\at{\pm 1}=0\,.
\end{align*}
Moreover, the nonlinear eigenvalue equation \eqref{Eqn:EV1} combined with 
$\chi_{\xi_\si} \cdot u_\si = \theta + \la_\si\,w_\si$ and $\si=\eta\,\la_\si$ yields
\begin{align}
\label{Eqn:App0}
\theta + \la_\si \,\ol{w}_\si 
=\ol{\calA}_{\xi_\si} \ol{w}_\si\,,
\end{align}
where the operator 
\begin{align}
\label{Eqn:App8}
\at{\ol{\calA}_\xi \ol{w}}\at{\ol{x}}:=
\int\limits_{-1}^{+1} \xi\,a\bat{\xi\,\ol{x}-\xi\,\ol{y}}\,\ol{w}\at{\ol{y}}\dint{\ol{y}}
\end{align}
is the rescaled counterpart of $\calA_\xi$ from \eqref{Eqn:DefOpA}. 
%
%
\paragraph{Small eigenvalues for $\zeta=0$}
%
%
Proposition \ref{Prop:KRMon} shows $\xi_\si\approx0$ and by Taylor expansion of \eqref{Eqn:App8} we get
\begin{align*}
\bat{\ol{\calA}_\xi \ol{w}}\at{\ol{x}}&= 
\xi\cdot\at{a\at{0}\int\limits_{-1}^{+1} \ol{w}\at{\ol{y}}\dint{\ol{y}}}-\xi^3\cdot \at{\tfrac12\abs{a^{\prime\prime}\at{0}}
\int\limits_{-1}^{+1} \ol{w}\at{\ol{y}}\ol{y}^2\dint{\ol{y}} }\\&\qquad \qquad -\ol{x}^2\cdot\xi^3\cdot\at{
\tfrac12\abs{a^{\prime\prime}\at{0}}
\int\limits_{-1}^{+1} \ol{w}\at{\ol{y}}\dint{\ol{y}}}+\DO{\xi^5}\,.
\end{align*}
\begin{figure}[t!]%
\centering{%
\includegraphics[width=0.71\textwidth]{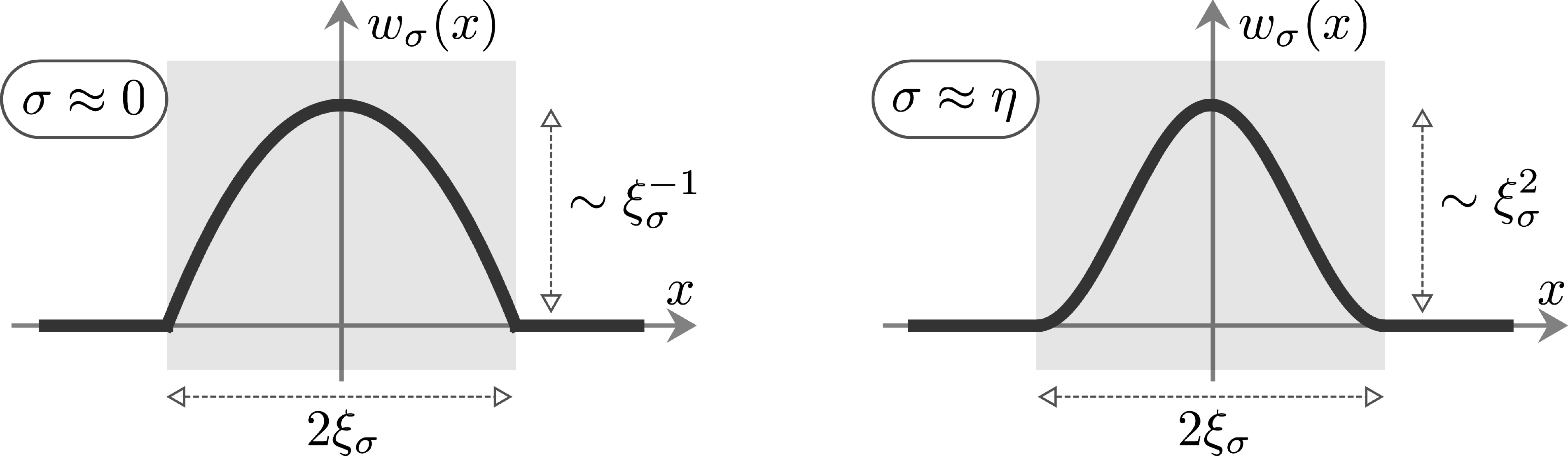}
}%
\caption{%
Schematic representation of the approximate formulas \eqref{App.Eqn4} 
(\emph{left panel}, $\xi_\si\sim\si^{1/3}$) and \eqref{App.Eqn5} (\emph{right panel}, $\xi_\si\sim\at{\eta-\sigma}^{-1}$), which describe the asymptotics of $u_\si$ in the special case $\zeta=0$. See also \eqref{Eqn:App6} and \eqref{Eqn:App3} as well as Conjectures \ref{Conj:A1} and \ref{Conj:A2}.%
}%
\label{Fig.LimitShapes}%
\end{figure}%
The eigenvalue equation \eqref{Eqn:App0} thus implies the approximate identity
\begin{align*}
\ol{w}_\si\at{\ol{x}}\approx c_\si\,\at{1-\ol{x}^2}
\end{align*}
for some constant $c_\si$ and any $\ol{x}\in\ccinterval{-1}{+1}$. Computing the integrals we find 
\begin{align*}
\theta +\la_\si\,c_\si{}\,\bat{1-\ol{x}^2}\approx c_\si\,
\xi_{\si}\,\Bat{\tfrac43\, a\at{0}}-c_{\si}\,\xi_\si^3\,\Bat{\tfrac{2}{15}\,\abs{a^{\prime\prime}\at{0}}}-\ol{x}^2\, c_{\si}\,\xi_\si^3\,\Bat{\tfrac{2}{3}\,\abs{a^{\prime\prime}\at{0}}}\,.
\end{align*}
and equating the coefficients --- first in front of $\ol{x}^2$ and afterwards in front of $1$ --- we identify the scaling relations
\begin{align*}
\la_\si \approx \tfrac23\,\abs{a^{\prime\prime}\at{0}}\,\xi_\si^3\,,\qquad \tfrac{4}{3}\,a\at{0}\,c_\si\,\xi_\si\approx \theta\,,
\end{align*}
where $\la_\si$ satisfies \eqref{Prop:ExUni.PEqn8}. Moreover, the function 
\begin{align}
\label{App.Eqn4}
w_\si\at{x}\approx \frac{3\,\theta}{4\, a\at{0}\, \xi_\si}\,\chi_{\xi_\si}\at{x}\,\at{1-\xi_\si^{-2}x^2}\
\end{align} 
approaches a Dirac distribution as $\si\to0$, see also Figure \ref{Fig.LimitShapes}, so \eqref{Eqn:App3} implies that $u_\si$ converges to a multiple of $a$, where the scaling factor is consistent with $u_\si\at{\xi_\si}=\theta$. In summary, we expect the following behavior for small $\si$.
\begin{conjecture}[asymptotics for $\zeta=0$ and $\si\approx 0$] 
\label{Conj:A1}
Suppose that $a$ that is sufficiently smooth at the origin with $a^{\prime\prime}\at{0}<0$. Then, we have
\begin{align}
\label{Conj:A1.Eqn1a}
\si^{-1/3}{\xi_\si}\;\;\xrightarrow{\;\si\searrow0\;}\;\; \frac{3^{1/3}}{2^{1/3}\,\eta^{1/3}\abs{a^{\prime\prime}\at{0}}^{1/3}}
\end{align}
as well as
\begin{align}
\label{Conj:A1.Eqn1b}
\qquad 
\si^{1/3}\ol{w}_\si\at{\ol{x}}\;\;\xrightarrow{\;\si\searrow0\;}\;\;\frac{3^{2/3}\,\eta^{1/3}\,\theta\abs{a^{\prime\prime}\at{0}}^{1/3}}{2^{5/3}\,a\at{0}}\,\ol{\chi}\at{\ol{x}}\,\bat{1-\ol{x}^2}\,,
\end{align}
where $\ol{\chi}$ is the indicator function of $\ol{I}$. Moreover,
\begin{align}
\label{Conj:A1.Eqn2}
\lim_{\si\to0 }u_\si\at{x} = \frac{\theta\,a\at{x}}{a\at{0}}\,,\qquad \lim_{\si\to0 }f\bat{u_\si\at{x}} =0 
\end{align}
holds at least in the sense of pointwise convergence.
\end{conjecture}
We believe that the assertions of Conjecture \ref{Conj:A1} can be derived by standard arguments but notice that \eqref{Conj:A1.Eqn1a}
and \eqref{Conj:A1.Eqn1b} do not cover the kernels \eqref{Eqn:Kernel2} and \eqref{Eqn:Kernel3} since these are not smooth and satisfy $a^{\prime\prime}\at{0}=-\infty$ and $a^{\prime\prime}\at{0}=0$, respectively. Nonetheless, a modified asymptotic analysis should reveal that \eqref{Conj:A1.Eqn2} is still satisfied. For the tent map kernel,  a similar convergence result has been derived in \cite{TV14} for $\zeta>0$ and $\eta\to\infty$. This limit has much in common with the anticontinuum or high-energy limit studied in \cite{HM19b}.
%
%
\paragraph{Small eigenvalues for $\zeta>0$}
%
%
\begin{figure}[t!]%
\centering{%
\includegraphics[width=0.98\textwidth]{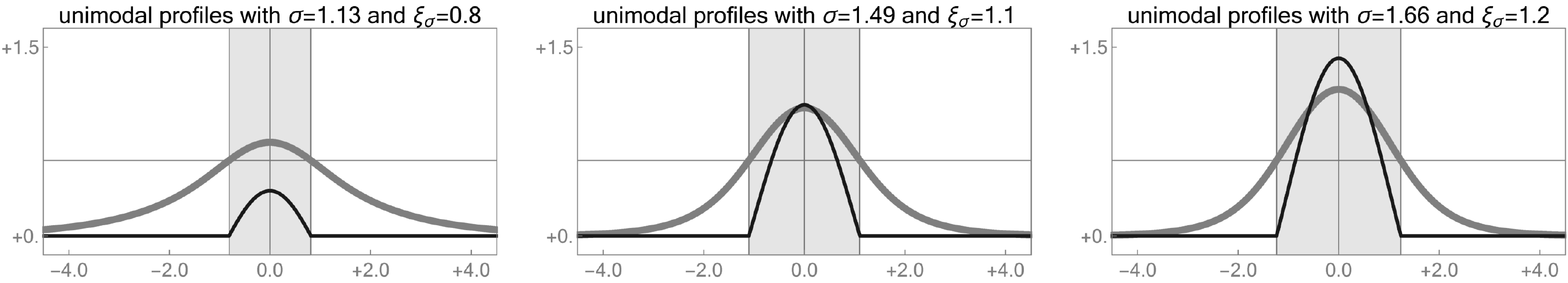}
}%
\caption{%
Numerical simulation with $\zeta=1.0$, $\theta=0.6$, $\eta=2.5$ for the Gaussian kernel \eqref{Eqn:Kernel1}. The computations are performed with the nonlinear improvement dynamics described in \cite{HM20} and the profiles $u_\si$ (gray) and $\tilde{f}\at{u_\si}$ (black) are shown for several values of $\si$. Conjecture \ref{Conj:A3} predicts that $u_\si$ converges as $\si\searrow\zeta$ to the constant function with value $\theta$, but this limit is hard to capture numerically. %
}%
\label{Fig.Ex4U}%
\end{figure}%
In this case, we have to replace the kernel $a$ by the transformed kernel from Proposition \ref{Prop:Transformation}, which we now denote by $\tilde{a}_\si$ as it depends on $\si$, and $\si$ by $\tilde{\si}=\si-\zeta$. The problem is that $\tilde{a}_\si$ does not converge strongly as $\si\searrow \zeta$ but its amplitude gets smaller while its effective width approaches $\infty$, see Figure \ref{Fig.kernels} for an illustration. The precise analysis depends on the singular scaling behavior of $\tilde{a}_\si$. For sufficiently nice kernels we can suppose that the limits
\begin{align}
\label{Eqn:App10}
\ka_0:= \lim_{\si\searrow \zeta}\at{\si-\zeta}^{-1/2}\,\tilde{a}_\si\at{0}\,,\qquad  
\ka_2:= \lim_{\si\searrow \zeta}\at{\si-\zeta}^{-1}\abs{\tilde{a}_\si^{\prime\prime}\at{0}}
\end{align}
are well defined, and direct computations for the Gaussian kernel \eqref{Eqn:Kernel1} provide the values $\ka_0=1$ as well as $\ka_2 = 2\,\zeta\at{3/2}/\sqrt{\pi}$ in terms of the zeta function.
\par
In consistency with Proposition \ref{Prop:Transformation} we base our analysis of
\begin{align*}
\ol{\tilde{w}}_\si\at{\ol{x}} = \tilde{w}_\si\at{\xi_\si\ol{x}}
\,,\qquad 
\tilde{w}_\si\at{x} = \frac{\tilde{f}\at{u_\si\at{x}}}{\tilde{\si}} \,,
\end{align*}
which solves an equation similar to \eqref{Eqn:App0}. The scaling relations encoded by \eqref{Eqn:App10} still allow us to Taylor expand $\tilde{a}_\si$ even though the corresponding cut-off parameter $\xi_\si$ can no longer assumed to be small. In fact, repeating the arguments from above we find
\begin{align*}
\ol{\tilde{w}}\at{\ol{x}}=\tilde{c}_\si\,\at{1-\ol{x}^2}
\end{align*}
as well as
\begin{align*}
\theta +\eta^{-1}\at{\si-\zeta}\,\tilde{c}_\si{}\,\bat{1-\ol{x}^2}\approx \tfrac43\, \ka_0\,\tilde{c}_\si\,
\xi_{\si}\,\,\at{\si-\zeta}^{1/2}-\tfrac{2}{15}\,\ka_2\,\tilde{c}_{\si}\,\xi_\si^3\,\at{\si-\zeta}-\ol{x}^2\, \tfrac23\,\ka_2\,\tilde{c}_{\si}\,\xi_\si^3\,\at{\si-\zeta}\,.
\end{align*}
Equating coefficients gives
\begin{align*}
\tfrac23\,\ka_2\,\xi_\si^3\approx \eta^{-1}\,,\qquad 
\tfrac43\, \ka_0\,\tilde{c}_\si\,
\xi_{\si}\,\,\at{\si-\zeta}^{1/2}\approx \theta
\end{align*}
as well as 
\begin{align*}
\at{\si-\zeta}^{1/2}\,\tilde{w}_\si\approx \frac{3\,\theta}{4\,\ka_0\,\xi_\si}\,\chi_{\xi}\at{x}\,\bat{1-\xi_\si^{-2}\,x^2}\,,
\end{align*}
and in summary we find a limiting behavior that differs considerably from Conjecture \ref{Conj:A1}. In particular, $\at{\si-\zeta}^{1/2}\tilde{w}_\si$
converges to a piecewise smooth limit function with integral $\ka_0^{-1}\,\theta$, whose convolution with the small and slowly varying function $\tilde{a}_\si$ approximates $u_\si$ up to a small prefactor. See also Figure \ref{Fig.Ex4U}.

\begin{conjecture}[asymptotics for $\zeta>0$ and $\si\approx \zeta$] 
\label{Conj:A3}
Suppose that the kernel $a$ is sufficiently regular so that \eqref{Eqn:App10} is satisfied. Then, we have
\begin{align}
\notag
\xi_\si\;\;\xrightarrow{\;\si\searrow0\;}\;\; \frac{3^{1/3}}{2^{1/3}\,\eta^{1/3}\,\ka_2^{1/3}}\,,\qquad 
\at{\si-\zeta}^{1/2}\ol{w}_\si\at{\ol{x}}\;\;\xrightarrow{\;\si\searrow0\;}\;\;\frac{3^{2/3}\,\eta^{1/3}\,\theta\,\ka_2^{1/3}}{2^{5/3}\,\ka_0}\,\ol{\chi}\at{\ol{x}}\,\bat{1-\ol{x}^2}
\end{align}
as well as
\begin{align}
\notag
\lim_{\si\to0 }u_\si\at{x} = \theta\,,\qquad \lim_{\si\to0 }\tilde{f}\bat{u_\si\at{x}} =0 
\end{align}
in the sense of pointwise convergence.
\end{conjecture}
The different limit behaviors for $\zeta=0$ and $\zeta>0$ can also be understood heuristically as follows. For $\zeta=\si=0$, there exists a plethora of nonlinear eigenfunction in $\calU$, namely any function $u\in\calU$ which satisfies $f\at{u}\equiv0$ due to $\norm{u}_\infty=u\at{0}\leq\theta$. In the case of $\zeta>0$, however, $\si=\zeta$ combined with $0\leq u\leq\theta$ reduces \eqref{Eqn:EV1} to $u= a\ast u$, but the only fixed points of the convolution operator are the constant functions.
\paragraph{Large eigenvalues}
%
%
We start again with $\zeta=0$ but it turns out that our formal asymptotic results cover
the case $\zeta>0$ as well. Since $\xi_\si$ is now large, we restate \eqref{Eqn:App8} as
\begin{align*}
\at{\ol{\calA}_\xi \ol{w}}\at{\ol{x}}=
\int\limits_{\xi\at{\ol{x}-1}}^{\xi\at{\ol{x}+1}} a\bat{y}\,\ol{w}\at{\ol{x}+\xi^{-1}y}\dint{y}
\end{align*} 
and employ the formal asymptotic expansion
\begin{align*}
\at{\ol{\calA}_\xi \ol{w}}\at{\ol{x}}\approx
\ol{w}\at{\ol{x}}-m\,\xi^{-2}\,\ol{w}^{\prime\prime}\at{x}
\int\limits_{\xi\at{\ol{x}-1}}^{\xi\at{\ol{x}+1}} a\bat{y}\,\ol{w}\at{\ol{x}+\xi^{-1}y}\dint{y}\,.
\end{align*} 
This formula involves
\begin{align*}
m:=-\tfrac12\int\limits_{-\infty}^{\infty} y^2 a\at{y}\dint{y}
\end{align*}
and holds for any fixed $-1<\ol{x}<+1$ thanks to Assumption \ref{Ass:Kernel}. In combination with \eqref{Eqn:App0} we thus obtain the approximate ODE
\begin{align*}
\theta + \la_\si\, \ol{w}_\si\approx\ol{w}_\si - \xi_\si^{-2}\,m\, \ol{w}^{\prime\prime}\,,
\end{align*}
which admits the consistent solution
\begin{align}
\label{Eqn:App1}
\la_\si \approx 1-\xi_\si^{-2}\,\pi^2\,m\,,\qquad \ol{w}_\si\at{\ol{x}}\approx d_\si \,\bat{1+\cos\at{\pi \ol{x}}}\,,\qquad d_\si = \xi_\si^2\frac{\theta}{\pi^2\,m}\,.
\end{align}
These relation imply
\begin{align}
\label{App.Eqn5}
w_\si\at{x}\approx \frac{\theta\,\xi_\si^2}{\pi^2\,m}\,\chi_{\xi_\si}\at{x}\,\bat{1+\cos\at{\pi\,\xi_\si^{-1}\,x }}
\end{align}
and give rise to the following claim.  
\begin{conjecture}[asymptotics for $\zeta=0$ and $\si\approx\eta $] 
\label{Conj:A2}%
For any kernel $a$ that decays sufficiently fast at infinity, we  have
\begin{align}
\label{Conj:A2.Eqn1}
\at{\eta-\si}^{1/2}{\xi_\si}\;\;\xrightarrow{\;\si\nearrow\eta\;}\;\; \pi\,m^{1/2}\,\eta^{1/2}\,,\qquad 
\at{\eta-\si}\ol{w}_\si\at{\ol{x}}\;\;\xrightarrow{\;\si\nearrow\eta\;}\;\;\eta\,\theta\,\ol{\chi}\at{\ol{x}}\,\at{1-\cos\at{\pi\ol{x}}}
\end{align}
and hence $\norm{u_\si}_p\sim\at{\eta-\si}^{-1/2-1/\at{2p}}$ for any $p\in\ccinterval{1}{\infty}$. 
\end{conjecture}
The limit $\sigma\nearrow\zeta+\eta$ in case of $\zeta>0$ involves again the transformed kernel from Proposition \ref{Prop:Transformation.Eqn1}, which satisfy $\tilde{a}_\si\to\tilde{a}_\eta$. We therefore expect that \eqref{Conj:A2.Eqn1} remains true provided that $\si-\eta$ is replaced by $\tilde{\si}-\eta=\si-\zeta-\eta$ and $m$ is computed with $\tilde{a}_\eta$ instead of $a$. 
\par
We further emphasize that the scaling relations \eqref{Conj:A2.Eqn1} are consistent with our numerical simulations in Figures \ref{Fig.Ex1U} and \ref{Fig.Ex1N}. The analytical justification of Conjecture \ref{Conj:A2} and the underlying approximation  \eqref{Eqn:App1} is much harder than the rigorous derivation of \eqref{Conj:A1.Eqn1a} and \eqref{Conj:A1.Eqn1b} because the Taylor expansion is now  applied to $\ol{w}_\si$ and  requires uniform estimates of $\norm{\ol{w}_\si^{\prime\prime}}_p$ for at least one $p\in\ccinterval{1}{\infty}$. A similar problem concerns the linear eigenvalue problem from \S\ref{sect:lin}. In order to ensure the consistent approximation
\begin{align*}
v_\xi\at{x}\approx -\at{\xi}^{-1/2}\,\chi_{\xi}\at{x}\,\sin\bat{\xi^{-1}\pi x }
\end{align*}
for the normalized eigenfunction from Proposition \ref{Prop:KR}, we have to guarantee that
\begin{math}
\nnorm{v_\xi^{\prime\prime}}_p\leq C_p\,\xi^{\frac{1}{p}-\frac{5}{2}}
\end{math}  
holds for all large $\xi$ and some constant $C_p$ independent of $\xi$. We are, however, not  ware of a corresponding reference. Moreover, the limit $\xi\to \infty$ is rather intricate due to a huge number of nearby eigenvalues that discretize the continuous spectrum of the limit operator $\calA_\infty w=a\ast w$.
%
%
%
%
%

\end{document}